\documentclass[11pt,a4paper,reqno]{amsart}
\usepackage{amssymb}
\usepackage{latexsym}
\usepackage{amsmath}
\usepackage{mathrsfs}
\usepackage{cite}
\usepackage{textcomp}
\usepackage{upgreek}
\usepackage{mathabx}
\usepackage{amscd}
\usepackage{color}
\usepackage{cancel}
\usepackage{calc}
\usepackage{comment}
\usepackage{enumerate}
\usepackage{tikz}

\newcommand{\scal}[2]{\langle #1,#2\rangle}

\newcommand{\nn}[1]{\mathbf N^{#1}}
\newcommand{\rr}[1]{\mathbf R^{#1}}

\newcommand{\ro}{\mathbf R}
\newcommand{\co}{\mathbf C}
\newcommand{\no}{\mathbf N}

\newcommand{\ep}{\varepsilon}

\newcommand{\dd}{\mathrm {d}}
\newcommand{\cdo}{\, \cdot \, }

\newcommand{\eabs}[1]{\langle #1\rangle}     %%%%%   for <x>
\newcommand{\cF}{\mathscr{F}}

\newcommand{\cS}{\mathscr{S}}
\newcommand{\leqs}{\leqslant}
\newcommand{\geqs}{\geqslant}
\newcommand{\la}{\langle}
\newcommand{\ra}{\rangle}

\newcommand{\wh}{\widehat}

\numberwithin{equation}{section}          %Detta goer att man faar
\newtheorem{thm}{Theorem}
\numberwithin{thm}{section}
\newtheorem{prop}[thm]{Proposition}
\newtheorem{cor}[thm]{Corollary}
\newtheorem{lemma}[thm]{Lemma}

\newcommand{\rubrik}{}

\theoremstyle{definition}

\theoremstyle{remark}
\newtheorem{rem}[thm]{Remark}

\title[Polynomially oscillatory multipliers on Gelfand--Shilov spaces]
{Polynomially oscillatory multipliers on Gelfand--Shilov spaces}

\author[A. Arias Junior]{Alexandre Arias Junior}

\address{Department of Computer Science and Mathematics (DCM - FFCLRP), University of S\~ao Paulo (USP), Ribeir\~ao Preto, SP, 14040-901, Brazil}

\email{alexandre.ariasjunior[AT]usp.br}

\author[P. Wahlberg]{Patrik Wahlberg}

\address{Dipartimento di Scienze Matematiche, Politecnico di Torino, Corso Duca degli Abruzzi 24, 10129 Torino, Italy}

\email{patrik.wahlberg[AT]polito.it}

\keywords{Gelfand--Shilov spaces, oscillatory multiplier operators, linear evolution equations, well-posedness}

\subjclass[2020]{Primary: 42B35, 35B65, 47D03, 46F05, 46E10.
\quad Secondary: 35Q40, 81Q05}

\begin{document}

\maketitle

\begin{abstract}
We study continuity of the multiplier operator $e^{i q}$
acting on Gelfand--Shilov spaces, 
where $q$ is a polynomial on $\rr d$ of degree at least two with real coefficients. 
In the parameter quadrant for the spaces we identify a wedge that depends on the polynomial degree
for which the operator is continuous. 
We also show that in a large part of the complement region
the operator is not continuous in dimension one. 
The results give information on well-posedness for linear evolution equations
that generalize the Schr\"odinger equation for the free particle. 
\end{abstract}

%%%%%%%%%%%%%%%%%%%%%%
\section{Introduction}\label{sec:intro}
%%%%%%%%%%%%%%%%%%%%%%

Let $q$ be a polynomial on $\rr d$ with real coefficients.
We consider in this paper the multiplication operator $T = T_g$ with $g(x) = e^{i q(x)}$ defined by
\begin{equation}\label{eq:multiplier1}
(T f) (x) = e^{i q(x)} f(x), \quad x \in \rr d. 
\end{equation}

This operator is unitary on $L^2(\rr d)$ and obviously acts continuously on the Schwartz space $\cS(\rr d)$ of smooth functions whose derivatives decay rapidly.  
The Schwartz space can equivalently be defined as the space of functions $f$ that 
decay superpolynomially at infinity, plus the same condition on the Fourier transform $\wh f$. 

Our goal is to sort out for which parameters of Gelfand--Shilov spaces 
the operator $T$ is continuous. 
The Gelfand--Shilov spaces are scales of spaces smaller than the Schwartz space. 
In fact a Gelfand--Shilov space has two parameters $\theta,s > 0$ and can be defined by the exponential decay 
conditions
\begin{equation}\label{eq:gelfandshilovcond}
\sup_{x \in \rr d} e^{a |x|^{\frac{1}{\theta}}} |f(x)| < \infty, \quad
\sup_{\xi \in \rr d} e^{a | \xi |^{\frac{1}{s}}} | \wh f(\xi)| < \infty, 
\end{equation}
for some $a > 0$. 
The parameter $\theta$ thus controls the decay rate of $f$, and 
the parameter $s$ controls the decay rate of $\wh f$, that is the smoothness of $f$. 

We use two types of Gelfand--Shilov spaces: The Roumieu spaces $\mathcal S_\theta^s(\rr d)$
for which \eqref{eq:gelfandshilovcond} holds for some $a > 0$, and the Beurling spaces 
$\Sigma_\theta^s(\rr d)$
for which \eqref{eq:gelfandshilovcond} holds for all $a > 0$. 

Our first result concerns sufficient conditions for continuity on Gelfand--Shilov spaces. 

\begin{thm}\label{thm:cont1}
Define $T$ by \eqref{eq:multiplier1}
where $q$ is a polynomial on $\rr d$ with real coefficients and degree $m \geqs 2$, 
and let $s,\theta > 0$. 

\begin{enumerate}[\rm (i)]

\item If $s \geqs (m-1) \theta \geqs 1$ then $T$ is continuous on $\mathcal S_\theta^s(\rr d)$.

\item If $s \geqs (m-1) \theta \geqs 1$ and $(\theta,s) \neq \left( \frac1{m-1},1 \right)$ then $T$ is continuous on $\Sigma_\theta^s(\rr d)$.

\end{enumerate}

\end{thm}

An immediate consequence is the corresponding claim for ultradistributions $\left( \mathcal S_\theta^s \right)'(\rr d)$
and $\left( \Sigma_\theta^s \right)'(\rr d)$, see Corollary \ref{cor:contdual}.

Secondly we prove negative results for $d = 1$ in a parameter region which is close to complementary to $s \geqs (m-1) \theta$. 
The first result generalizes \cite[Proposition~2]{Arias1}. 

\begin{thm}\label{thm:discont1}
Define $T$ by \eqref{eq:multiplier1}
where $q$ is a polynomial on $\ro $ with real coefficients and degree $m \geqs 2$, 
and let $s,\theta > 0$. 
If 
\begin{equation*}
1 \leqs s < \theta m - \max(\theta,1),
\end{equation*}
and $\theta \geqs 1$ if $m = 3$, 
then $T \mathcal S_\theta^s(\ro) \nsubseteq \mathcal S_\theta^s(\ro)$. 
\end{thm}

\begin{thm}\label{thm:discont2}
Define $T$ by \eqref{eq:multiplier1}
where $q$ is a polynomial on $\ro $ with real coefficients and degree $m \geqs 2$, 
and let $s,\theta > 0$. 
If 
\begin{equation*}
1 < s < \theta m - \max(\theta,1),
\end{equation*}
and $\theta > 1$ if $m = 3$, 
then $T \Sigma_\theta^s(\ro) \nsubseteq \Sigma_\theta^s(\ro)$. 
\end{thm}

Note that the statements in Theorems \ref{thm:discont1} and \ref{thm:discont2} are stronger than the lack of continuity
on the spaces $\mathcal S_\theta^s(\ro)$ and $\Sigma_\theta^s(\ro)$ respectively. 
We illustrate Theorems \ref{thm:cont1}, \ref{thm:discont1} and \ref{thm:discont2}
in Figure \ref{fig:parameter}.

Our results can be applied to well-posedness of the initial value Cauchy problem for 
linear evolution equations of the form 
\begin{equation}\label{eq:cauchyproblem}
\left\{
\begin{array}{rl}
\partial_t u(t,x) + i p(D_x) u (t,x) & = 0, \quad x \in \rr d, \quad t \in \ro, \\
u(0,\cdot) & = u_0 
\end{array}
\right.
\end{equation}
where $p: \rr d \to \ro$ is a polynomial with real coefficients of degree $m \geqs 2$. 
This generalizes the Schr\"odinger equation for the free particle where $m = 2$ and $p(\xi) = |\xi|^2$. 
The solution operator (propagator) to \eqref{eq:cauchyproblem} is 
\begin{equation*}
u (t,x) 
= e^{- i t p(D_x)} u_0 
= (2 \pi)^{- \frac{d}{2}} \int_{\rr d} e^{i \la x, \xi \ra - i t p(\xi)} \wh u_0 (\xi) \dd \xi 
\end{equation*}
for $u_0 \in \cS(\rr d)$. 
This means that the propagator
$e^{- i t p(D_x)} = \cF^{-1} T \cF$ is the conjugation by the Fourier transform
of the multiplier operator $(T f) (x) = e^{- i t p(x)} f(x)$ of the form \eqref{eq:multiplier1}. 
Since the Fourier transform maps the Gelfand--Shilov spaces into themselves with an interchange of indices 
as $\cF: \mathcal S_\theta^s(\rr d) \to \mathcal S_s^\theta(\rr d)$ and 
$\cF: \Sigma_\theta^s(\rr d) \to \Sigma_s^\theta(\rr d)$ we obtain the following 
consequence of Theorems \ref{thm:cont1}, \ref{thm:discont1} and \ref{thm:discont2}. 

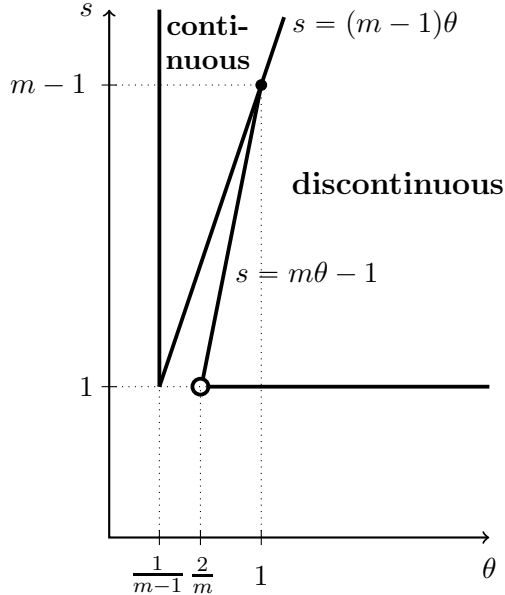
\begin{figure}\label{fig:parameter}
  \begin{tikzpicture}
    \draw [thick] [<->] (5,0) -- (0,0) --(0,7);
    \draw [line width=0.5mm] (1.2,2) circle (3pt);
    \draw [fill] (2,6) circle (2pt);
    \draw [-] (2,-0.1) -- (2,0.1);
    \node at (2,-0.5) {$1$};
    \node at (5,-0.4) {$\theta$};
    \draw [-] (-0.1,2) -- (0.1,2);
    \node at (-0.3,2) {$1$};
    \draw [-] (-0.1,6) -- (0.1,6);
    \node at (-0.8,6) {$m-1$};
    \node at (-0.3,7) {$s$};
    \draw [line width=0.5mm] (0.66,2) -- (2.3,6.9);
    \node at (3.5,6.8) {$s = (m-1) \theta$};
    \node at (2.6,3.5) {$s = m \theta - 1$};
    \node at (3.8,4.7) {\large \bf discontinuous};
    \node at (1.3,6.8) {\bf conti-};
    \node at (1.3,6.3) {\bf nuous};
    \draw [line width=0.5mm] (1.23,2.09) -- (2,6);
    \draw [line width=0.5mm] (1.28,2) -- (5,2);
    \draw [line width=0.5mm] (0.66,2) -- (0.66,7);
    \draw [dotted] (0.66,0) -- (0.66,2);
    \draw [dotted] (0,2) -- (2,2);
    \draw [dotted] (1.2,0) -- (1.2,2);
    \draw [dotted] (0,6) -- (2,6);
    \draw [dotted] (2,0) -- (2,6);
    \node at (0.64,-0.5) {$\frac1{m-1}$};
    \draw [-] (0.66,-0.1) -- (0.66,0.1);
    \draw [-] (1.2,-0.1) -- (1.2,0.1);
    \node at (1.23,-0.5) {$\frac2{m}$};
%    \fill[semitransparent,gray] (0.66,7) -- (0.66,2) -- (2.35,7);
  \end{tikzpicture}
  \caption{The $(\theta,s)$-parameter quadrant with behavior of the operator $T$ acting on Gelfand--Shilov spaces indicated when $m \geqs 4$ and $d = 1$.}
\end{figure}

\begin{cor}\label{cor:evolutioneq}
Let $p: \rr d \to \ro$ be a polynomial with real coefficients of degree $m \geqs 2$, 
consider the solution operator $e^{- i t p(D_x)}$ to the Cauchy problem \eqref{eq:cauchyproblem}, 
let $s,\theta > 0$ and let $t \in \ro$. 

\begin{enumerate}[\rm (i)]

\item If $\theta \geqs (m-1) s \geqs 1$ then $e^{- i t p(D_x)}$ is continuous on $\mathcal S_\theta^s(\rr d)$.

\item If $\theta \geqs (m-1) s \geqs 1$ and $(s,\theta) \neq \left( \frac1{m-1},1 \right)$ then $e^{- i t p(D_x)}$ is continuous on $\Sigma_\theta^s(\rr d)$.

\item If $d = 1$, $1 \leqs \theta < s m - \max(s,1)$, 
$s \geqs 1$ if $m = 3$, 
and $t \neq 0$
then $e^{- i t p(D_x)} \mathcal S_\theta^s(\ro) \nsubseteq \mathcal S_\theta^s(\ro)$. 

\item If $d = 1$, $1 < \theta < s m - \max(s,1)$, 
$s > 1$ if $m = 3$, 
and $t \neq 0$
then $e^{- i t p(D_x)} \Sigma_\theta^s(\ro) \nsubseteq \Sigma_\theta^s(\ro)$.  
\end{enumerate}
\end{cor}

The proof of Theorem \ref{thm:cont1} uses Debrouwere and Neyt's recent results \cite{Debrouwere1}
concerning characterization of smooth functions that acts continuously by multiplication on Gelfand--Shilov spaces, 
plus a result from \cite{Wahlberg1}. 
The proofs of Theorems \ref{thm:discont1} and \ref{thm:discont2}  
are based on ideas from the proof of 
\cite[Proposition~2]{Arias1} which concerns the multiplier function $e^{- i t x^2}$, 
together with an investigation of the polynomials $p_k$ that appear upon differentiation as 
$\partial^k e^{i q(x)} = p_{k} (x) \, e^{i q(x)}$ for $k \in \no$, where $q$ is a monomial.

The paper is organized as follows. 
In Section \ref{sec:prelim} we specify notation, conventions and background material. 
Then in Section \ref{sec:derivative} we work out the structure and estimates for the derivatives of 
exponential monomials in one variable of the form $e^{\lambda x^m/m}$ for $x \in \ro$, $m \in \no$, 
$m \geqs 2$ and $\lambda \in \co \setminus \{ 0 \}$. The results are used as tools for the negative results 
Theorems \ref{thm:discont1} and \ref{thm:discont2}. 
Finally Section \ref{sec:proofcont} is devoted to the proof of Theorem \ref{thm:cont1}
and Section \ref{sec:proofdiscont} to the proofs of Theorems \ref{thm:discont1} and \ref{thm:discont2}.

%%%%%%%%%%%%%%%%%%%%%%
\section{Preliminaries}\label{sec:prelim}
%%%%%%%%%%%%%%%%%%%%%%

%%%%%%%%%%%%%%%%%%%%%%%%%%%%%%%%
\subsection{Notation}\label{subsec:notation}
%%%%%%%%%%%%%%%%%%%%%%%%%%%%%%%%

The floor function is denoted $\lfloor x \rfloor$ for $x \in \ro$.
Multiindices $\alpha \in \nn d$ are measured with the $1$-norm $|\alpha| = \sum_{j=1}^d \alpha_j$, 
whereas vectors $x \in \rr d$ are measured with the $2$-norm $|x| = \left( \sum_{j=1}^d x_j^2 \right)^{\frac12}$. 
We use the bracket $\eabs{x} = \left( 1 + |x|^2 \right)^{\frac12}$ for $x \in \rr d$. 
Partial differential operators on $\rr d$ are denoted $\partial^\alpha$ for $\alpha \in \nn d$, 
and $D^\alpha = i^{-|\alpha|} \partial^\alpha$.
The normalization of the Fourier transform is
\begin{equation*}
 \cF f (\xi )= \widehat f(\xi ) = (2\pi )^{-\frac d2} \int _{\rr
{d}} f(x)e^{-i\scal  x\xi }\, \dd x, \qquad \xi \in \rr d, 
\end{equation*}
for $f\in \cS(\rr d)$ (the Schwartz space), where $\scal \cdo \cdo$ denotes the scalar product on $\rr d$. 
The conjugate linear action of a (ultra-)distribution $u$ on a test function $\phi$ is written $(u,\phi)$, consistent with the $L^2$ inner product $(\cdo ,\cdo ) = (\cdo ,\cdo )_{L^2}$ which is conjugate linear in the second argument. 

%%%%%%%%%%%%%%%%%%%%%%%%%%%%%%%%
\subsection{Gelfand--Shilov spaces}\label{subsec:gelfangshilov}
%%%%%%%%%%%%%%%%%%%%%%%%%%%%%%%%

The Schwartz space $\cS (\rr d)$ is the subspace of the smooth functions for which the seminorms
\begin{equation}\label{eq:schwartzconst}
f \mapsto \sup_{x \in \rr d} |x^\alpha D ^\beta f(x)| := C_{\alpha \beta}
\end{equation}
are finite for all $\alpha, \beta \in \nn d$. 

In this paper we work with Gelfand--Shilov spaces and their dual ultradistribution spaces \cite{Gelfand2}. 
For Gelfand--Shilov spaces you impose certain restrictions of the constants $C_{\alpha \beta}$ in \eqref{eq:schwartzconst}
which leads to spaces that are smaller than $\cS (\rr d)$.

Let $\theta, s, h > 0$. 
The space denoted $\mathcal S_{\theta,h}^s(\rr d)$
is the set of all $f\in C^\infty (\rr d)$ such that
\begin{equation}\label{eq:seminorms1}
\| f \|_{\mathcal S_{\theta,h}^s} \equiv \sup \frac {|x^\alpha D ^\beta
f(x)|}{h^{|\alpha + \beta |} \alpha !^\theta \, \beta !^s}
\end{equation}
is finite, where the supremum is taken over all $\alpha ,\beta \in
\mathbf N^d$ and $x\in \rr d$.
The function space $\mathcal S_{\theta,h}^s$ is a Banach space which increases
with $h$, $s$ and $\theta$, and $\mathcal S_{\theta,h}^s \subseteq \cS$.
The topological dual $(\mathcal S_{\theta,h}^s)'(\rr d)$ is
a Banach space which contains the tempered distributions: $\cS'(\rr d) \subseteq (\mathcal S_{\theta,h}^s)'(\rr d)$.

The Beurling type \emph{Gelfand--Shilov space}
$\Sigma _\theta^s(\rr d)$ is the projective limit 
of $\mathcal S_{\theta,h}^s(\rr d)$ with respect to $h$ \cite{Gelfand2}. This means
\begin{equation}\label{GSspacecond1}
\Sigma _\theta^s(\rr d) = \bigcap _{h>0} \mathcal S_{\theta,h}^s(\rr d)
\end{equation}
and the Fr{\'e}chet space topology of $\Sigma _\theta^s (\rr d)$ is defined by the seminorms $\| \cdot \|_{\mathcal S_{\theta,h}^s}$ for $h>0$.
 
We have $\Sigma _\theta^s(\rr d)\neq \{ 0\}$ if and only if $s + \theta > 1$ \cite{Petersson1}. 
The topological dual of $\Sigma _\theta^s(\rr d)$ is the space of (Beurling type) \emph{Gelfand--Shilov ultradistributions} \cite[Section~I.4.3]{Gelfand2}
\begin{equation}\tag*{(\ref{GSspacecond1})$'$}
(\Sigma _\theta^s)'(\rr d) =\bigcup _{h>0} (\mathcal S_{\theta,h}^s)'(\rr d).
\end{equation}

The Roumieu type Gelfand--Shilov space is the union 
\begin{equation*}
\mathcal S_\theta^s(\rr d) = \bigcup _{h>0}\mathcal S_{\theta,h}^s(\rr d)
\end{equation*}
equipped with the inductive limit topology \cite{Schaefer1}, that is 
the strongest topology such that each inclusion $\mathcal S_{\theta,h}^s(\rr d) \subseteq\mathcal S_\theta^s(\rr d)$
is continuous. 
We have $\mathcal S _\theta^s(\rr d)\neq \{ 0\}$ if and only if $s + \theta \geqs 1$ \cite{Gelfand2}. 
The corresponding (Roumieu type) Gelfand--Shilov ultradistribution space is 
\begin{equation*}
(\mathcal S_\theta^s)'(\rr d) = \bigcap _{h>0} (\mathcal S_{\theta,h}^s)'(\rr d). 
\end{equation*}
For every $s, \theta > 0$ such that $s + \theta > 1$, and for any $\ep > 0$ we have
\begin{equation}\label{eq:GSinclusions}
\Sigma _\theta^s (\rr d)\subseteq \mathcal S_\theta^s(\rr d)\subseteq
\Sigma _{\theta+\ep}^{s+\ep}(\rr d).
\end{equation}

The dual spaces $(\Sigma _\theta^s)'(\rr d)$ and $(\mathcal S _\theta^s)'(\rr d)$ may be equipped with several topologies. 
In this paper we use the weak$^*$ topologies on $(\Sigma _\theta^s)'(\rr d)$ and $(\mathcal S _\theta^s)'(\rr d)$ respectively. 

The Gelfand--Shilov (ultradistribution) spaces enjoy invariance properties, with respect to 
translation, dilation, tensorization, coordinate transformation and (partial) Fourier transformation.
The Fourier transform extends 
uniquely to homeomorphisms on $\mathscr S'(\rr d)$, from $(\mathcal
S_\theta^s)'(\rr d)$ to $(\mathcal
S_s^\theta)'(\rr d)$, and from $(\Sigma _\theta^s)'(\rr d)$ to $(\Sigma _s^\theta)'(\rr d)$, and restricts to 
homeomorphisms on $\mathscr S(\rr d)$, from $\mathcal S_\theta^s(\rr d)$ to $\mathcal S_s^\theta(\rr d)$, 
and from $\Sigma _\theta^s(\rr d)$ to $\Sigma _s^\theta(\rr d)$, and to a unitary operator on $L^2(\rr d)$.

Chung, Chung and Kim characterized in \cite{Chung1} 
the Roumieu space $\mathcal S_\theta^s(\rr d)$ as the space of functions $f \in C^\infty(\rr d)$
that satisfy \eqref{eq:gelfandshilovcond} for some $a > 0$. 
It also follows that the Beurling space $\Sigma_\theta^s(\rr d)$ can be characterized as the 
space of functions $f \in C^\infty(\rr d)$ that satisfy \eqref{eq:gelfandshilovcond} for all $a > 0$. 

We need the following result which says that we may use an alternative family of seminorms instead of the seminorms \eqref{eq:seminorms1} indexed by $h > 0$ for the spaces $\Sigma _\theta^s(\rr d)$ and $\mathcal S _\theta^s(\rr d)$. 
This is the family of seminorms
\begin{equation}\label{eq:seminorms2}
\| f \|_a \equiv \sup_{\beta \in \nn d, \ x \in \rr d} e^{a |x|^{\frac{1}{\theta}}} \beta !^{-s}  a^{|\beta |} |D ^\beta f(x)| 
\end{equation}
indexed by $a > 0$. 
The result may be considered quite well known but we write down a proof for the benefit of the reader. 

\begin{lemma}\label{lem:seminormequiv}
Suppose $\theta, s > 0$ and $\theta + s > 1$. 
For any $a > 0$ there exists $C, h > 0$ such that 
\begin{equation}\label{eq:snestimate1}
\| f \|_a \leqs C \| f \|_{\mathcal S_{\theta,h}^s}, \quad f \in \Sigma_\theta^s (\rr d). 
\end{equation}
For any $h > 0$ there exists $C, a > 0$ such that 
\begin{equation}\label{eq:snestimate2}
\| f \|_{\mathcal S_{\theta,h}^s} \leqs C \| f \|_a, \quad f \in \Sigma_\theta^s (\rr d). 
\end{equation}
\end{lemma}

\begin{proof}
Let $f \in \Sigma_\theta^s(\rr d)$. 
From \eqref{eq:seminorms1} we have for any $h>0$
\begin{equation*}
|x^\alpha D^\beta f(x)| \leqs \| f \|_{\mathcal S_{\theta,h}^s} \alpha!^\theta \beta!^s h^{|\alpha+\beta|}, \quad \alpha, \beta \in \nn d, \quad x \in \rr d.
\end{equation*}
This gives for any $n \in \no$ and any $\beta \in \nn d$
\begin{equation*}
|x|^n |D^\beta f(x)| 
\leqs d^{\frac{n}{2}} \max_{|\alpha|=n} |x^\alpha D^\beta f(x)|
\leqs d^{\frac{n}{2}}
\| f \|_{\mathcal S_{\theta,h}^s} n!^\theta \beta!^s h^{n + |\beta|}, \quad x \in \rr d. 
\end{equation*}
Thus for $a > 0$ we have 
\begin{align*}
\exp\left(\frac{a}{\theta}|x|^{\frac{1}{\theta}} \right) |D^\beta f(x)|^{\frac{1}{\theta}}
& = \sum_{n=0}^\infty \frac{|x|^{\frac{n}{\theta}} |D^\beta f(x)|^{\frac{1}{\theta}} \left( d^{\frac12}  h \right)^{-\frac{n}{\theta}}}{n!}  
\left(\frac{a \left( d^{\frac12}  h \right)^{\frac{1}{\theta}} }{\theta}\right)^n \\
& \leqs \| f \|_{\mathcal S_{\theta,h}^s}^{\frac{1}{\theta}} \beta!^{\frac{s}{\theta}} h^{\frac{|\beta|}{\theta}} \sum_{n=0}^\infty 2^{-n}, \quad x \in \rr d, \quad \beta \in \nn d, 
\end{align*}
provided $2 a \left( d^{\frac12}  h \right)^{\frac{1}{\theta}} \leqs \theta$.
Hence
\begin{align*}
e^{ a |x|^{\frac{1}{\theta}} } |D^\beta f(x)|
& \leqs 2^\theta \| f \|_{\mathcal S_{\theta,h}^s} \beta!^s h^{|\beta|}
\end{align*}
which gives 
\begin{equation*}
\| f \|_a \leqs 2^\theta \| f \|_{\mathcal S_{\theta,h}^s}
\end{equation*}
provided $h \leqs \min\left( a^{-1}, d^{- \frac12} \left( \theta 2^{-1} a^{-1} \right)^\theta \right)$.
We have shown \eqref{eq:snestimate1}.

In order to show \eqref{eq:snestimate2} we let $f \in \Sigma_\theta^s(\rr d)$. 
From \eqref{eq:snestimate1} we know that $\| f \|_a < \infty$ for any $a > 0$. 
Hence we have for any $a > 0$, $\beta \in \nn d$ and $x \in \rr d$
\begin{align*}
\sum_{n=0}^\infty \frac{|x|^{\frac{n}{\theta}} |D^\beta f(x)|^{\frac{1}{\theta}}}{n!} \left(\frac{a}{\theta} \right)^n
& = 
e^{ \frac{a}{\theta} |x|^{\frac{1}{\theta}} }
|D^\beta f(x)|^{\frac{1}{\theta}} 
\leqs \| f \|_a^{\frac{1}{\theta}} \beta!^{\frac{s}{\theta}} a^{-\frac{|\beta|}{\theta}}, 
\end{align*}
which gives 
\begin{align*}
|x|^{n} |D^\beta f(x)| 
& \leqs \| f \|_a n!^\theta \beta!^s a^{-|\beta|} \left( \frac{\theta}{a} \right)^{\theta n}, \quad n \in \no, \quad \beta \in \nn d, \quad x \in \rr d, 
\end{align*}
and thus, using \cite[Eq.~(0.3.3)]{Nicola1}, 
\begin{align*}
|x^\alpha D^\beta f(x)| 
& \leqs \| f \|_a \alpha!^\theta \beta!^s a^{-|\beta|} \left( \frac{d \theta}{a} \right)^{\theta |\alpha|}, \quad \alpha, \beta \in \nn d, \quad x \in \rr d. 
\end{align*}
From this it follows that 
\begin{equation*}
\| f \|_{\mathcal S_{\theta,h}^s} \leqs \| f \|_a  , \quad f \in \Sigma_\theta^s(\rr d), 
\end{equation*}
for any $h > 0$ provided $a \geqs \max(h^{-1}, d \theta h^{-\frac{1}{\theta}})$. This proves \eqref{eq:snestimate2}. 
\end{proof}

%%%%%%%%%%%%%%%%%%%%%%%%%%%%%%%%
\subsection{Fa\` a di Bruno's formula}\label{subsec:faadibruno}
%%%%%%%%%%%%%%%%%%%%%%%%%%%%%%%%

Of the several available versions of Fa\` a di Bruno's formula we will use the following two. 
If $f,g \in C^\infty(\ro)$ then we have for any $k \in \no$
\begin{equation}\label{eq:faadibruno1}
\begin{aligned}
& \frac{\dd^k}{\dd x^k} \big( f (g(x) ) \big) \\
& = \sum_{m_1 + 2 m_2 + \cdots + k m_k = k} \frac{k!}{m_1! m_2! \cdots m_k!}
f^{(m_1 + \cdots + m_k)} (g(x) ) 
\prod_{j=1}^k \left( \frac{g^{(j)} (x)}{j!} \right)^{m_j}. 
\end{aligned}
\end{equation}
The second version of Fa\` a di Bruno's formula concerns $f: \ro \to \ro$ and $g: \rr d \to \ro$
with $f \in C^\infty(\ro)$ and $g \in C^\infty(\rr d)$, and reads 
\begin{equation}\label{eq:faadibruno2}
\begin{aligned}
\partial^\alpha \big( f (g(x) ) \big)
= \sum_{j = 1}^{|\alpha|} \frac{ f^{(j)}( g(x) )}{j!} \sum_{\overset{\alpha_1 + \cdots + \alpha_{j} = \alpha}{|\alpha_\ell | \geqs 1, \ 1 \leq \ell \leq j}} \frac{\alpha!}{\alpha_1! \cdots \alpha_j!} \prod_{\ell = 1}^{j} \partial^{\alpha_\ell} g(x)
\end{aligned}
\end{equation}
for any $\alpha \in \nn d \setminus \{ 0 \}$ \cite[Eq. (2.3)]{Gramchev1}.
For an even more general version of Fa\` a di Bruno's formula we refer the reader to \cite[Proposition~4.3]{Bierstone1}.

%%%%%%%%%%%%%%%%%%%%%%%%%%%%%%%%%%
\section{Derivatives of exponential monomials}\label{sec:derivative}
%%%%%%%%%%%%%%%%%%%%%%%%%%%%%%%%%%

Let $m \in \no$, $m \geqs 2$, $\lambda \in \co \setminus \{ 0 \}$, and consider the function 
\begin{equation}\label{eq:exppoly1}
g_\lambda (x) = e^{\lambda x^m/m}, \quad x \in \ro.  
\end{equation}
This function can be considered a generalized Gaussian. 
Clearly we have for $k \in \no$
\begin{equation}\label{eq:gaussianderivative}
\partial^k g_\lambda (x) = p_{\lambda,k} (x) \, g_\lambda (x)
\end{equation}
where $p_{\lambda,k}$ is a polynomial of degree $\deg p_{\lambda,k} = k (m-1)$. 
In fact we have $p_{\lambda,0} (x) = 1$, $p_{\lambda,1} (x) = \lambda x^{m-1}$, and the recursive relation 
\begin{equation}\label{eq:recursion1}
p_{\lambda,k+1} (x) = \lambda x^{m-1} p_{\lambda,k} (x) + p_{\lambda,k} ' (x)
\end{equation}
for $k \in \no$. 
When $m = 2 = - \lambda$ then $(-1)^k p_{\lambda,k}$ are the Hermite polynomials. 

\begin{lemma}\label{lem:derivativepolynomial}
Let $m \in \no$, $m \geqs 2$, $\lambda \in \co \setminus \{ 0 \}$, 
and let $g_\lambda$ be defined by \eqref{eq:exppoly1}
and $p_{\lambda,k}$ by \eqref{eq:gaussianderivative} for $k \in \no$. 
Then the polynomials $p_{\lambda,k}$ have the form 
\begin{equation}\label{eq:polynomial1}
p_{\lambda,k} (x) = \sum_{n = 0}^{\lfloor k \frac{m-1}{m} \rfloor} \lambda^{k-n} x^{ (m-1) k - n m} C_{k,n}
\end{equation}
for $k \in \no \setminus \{ 0 \}$, where $C_{k,n} \in \no \setminus \{ 0 \}$ for $0 \leqs n \leqs \lfloor k \frac{m-1}{m} \rfloor$, 
and $C_{k,0} = 1$ for all $k \in \no \setminus \{ 0 \}$.
\end{lemma}

\begin{proof}
First we note that for $k \geqs 1$ we have
\begin{equation}\label{eq:indexbounds}
\frac{k-1}2  \leqs \Big\lfloor \frac{k}2 \Big\rfloor \leqs \Big\lfloor k \frac{m-1}{m} \Big\rfloor \leqs k-1.
\end{equation}

If $k = 1$ then the sum \eqref{eq:polynomial1} contains only one term of the stated form as confirmed above, with $C_{1,0} = 1$.  
In an induction argument we suppose that \eqref{eq:polynomial1} holds true for a fixed $k \geqs 1$. 
By \eqref{eq:indexbounds} we have 
\begin{equation*}
\Big\lfloor (k+1) \frac{m-1}{m} \Big\rfloor 
\geqs \Big\lfloor \frac{k+1}2 \Big\rfloor \geqs 1
\end{equation*}
so the sum \eqref{eq:polynomial1} with $k$ replaced by $k+1$ does contain the term with index $n = 1$. 
We obtain from \eqref{eq:recursion1} and the induction hypothesis \eqref{eq:polynomial1} with $C_{k,0} = 1$   
\begin{equation}\label{eq:inductionstep}
\begin{aligned}
p_{\lambda,k+1} (x) 
& = \lambda x^{m-1} p_{\lambda,k} (x) + p_{\lambda,k} ' (x) \\
& = \lambda^{k+1} x^{ (m-1)(k+1)} 
+ \sum_{1 \leqs n \leqs  \lfloor k \frac{m-1}{m} \rfloor } \lambda^{k+1-n} x^{ (m-1) (k+1) - nm} C_{k,n} \\
& \quad + \lambda^{k} x^{ (m-1)k-1} (m-1)k \\
& \quad + \sum_{1 \leqs n <  k \frac{m-1}{m}} \lambda^{k-n} x^{ (m-1) k - nm - 1} C_{k,n} \Big( (m-1) k - nm \Big) \\
& = \lambda^{k+1} x^{ (m-1)(k+1)} 
+ \sum_{1 \leqs n \leqs  \lfloor k \frac{m-1}{m} \rfloor} \lambda^{k+1-n} x^{ (m-1) (k+1) - nm} C_{k,n} \\
& \quad + \lambda^{k} x^{ (m-1) (k+1) - m} (m-1) k \\
& \quad
+ \sum_{2 \leqs n <  k \frac{m-1}{m} + 1} \lambda^{k+1-n} x^{ (m-1) (k+1) - nm} C_{k,n-1} \Big( (m-1) k - (n-1)m \Big).
\end{aligned}
\end{equation}
Note that the term $\lambda^{k} x^{ (m-1) (k+1) - m} (m-1) k$ fits into \eqref{eq:polynomial1} with $k$ replaced by $k+1$ and index $n=1$. 
In the last sum the indices $n$ satisfy $2 m \leqs m n \leqs  (m-1) k + m - 1$ which gives
$2 \leqs n \leqs (k+1) \frac{ m - 1}{ m }$. Hence
\begin{equation*}
2 \leqs n \leqs \Big\lfloor (k+1) \frac{ m - 1}{ m } \Big\rfloor.
\end{equation*}

We may conclude that all terms in \eqref{eq:inductionstep} can be absorbed into the formula 
\eqref{eq:polynomial1} with $k$ replaced by $k+1$ for certain 
coefficients $C_{k+1,n} \in \no \setminus \{ 0 \}$ for $0 \leqs n \leqs \lfloor (k+1) \frac{ m - 1}{ m } \rfloor$. 
In fact the coefficients $C_{k+1,n}$ are linear combinations of $\{ C_{k,n} \}_{0 \leqs n \leqs \lfloor k \frac{ m - 1}{ m } \rfloor}$ with positive integer coefficients. 
It also follows that $C_{k+1,0} = 1$. 
This proves the induction step which guarantees that \eqref{eq:polynomial1} 
holds for any $k \in \no \setminus \{ 0 \}$, and $C_{k,0} = 1$ for all $k \in \no \setminus \{ 0 \}$. 
\end{proof}

In order to understand more about the polynomials $p_{\lambda,k}$ 
we would like to gain
knowledge about the coefficients $C_{k,n}$. 
We know from Lemma \ref{lem:derivativepolynomial} that $C_{k,0} = 1$ for all $k \in \no \setminus \{ 0 \}$. 

We need a simple lemma. 

\begin{lemma}\label{lem:multiple}
Suppose $m \in \no$, $m \geqs 2$ and $k \in \no \setminus \{ 0 \}$. 
Then 
\begin{equation}\label{eq:kmultm}
k \in m \no \qquad \Longrightarrow \qquad
\Big\lfloor (k+1) \frac{m-1}{m} \Big\rfloor = \Big\lfloor k \frac{m-1}{m} \Big\rfloor
\end{equation}
and
\begin{equation}\label{eq:knotmultm}
k \notin m \no \qquad \Longrightarrow \qquad
\Big\lfloor (k+1) \frac{m-1}{m} \Big\rfloor = \Big\lfloor k \frac{m-1}{m} \Big\rfloor + 1. 
\end{equation}
\end{lemma}

\begin{proof}
If $k \in m \no$ then there exists $p \in \no \setminus \{ 0 \}$ such that $k = m p$
which gives 
\begin{equation*}
k \frac{m-1}{m} = p(m-1) = \Big\lfloor k \frac{m-1}{m} \Big\rfloor
\end{equation*}
and 
$(k+1) \frac{m-1}{m} = p(m-1) + 1 - \frac1m$.
Since $\frac12 \leqs 1 - \frac1m < 1$ we get 
\begin{equation*}
\Big\lfloor (k+1) \frac{m-1}{m} \Big\rfloor = p(m-1) = \Big\lfloor k \frac{m-1}{m} \Big\rfloor
\end{equation*}
which proves \eqref{eq:kmultm}.

If instead $k \notin m \no$ then there exist $p,q \in \no$ with $1 \leqs q \leqs m-1$ such that $k = m p + q$. 
Then $(q+1)(m-1) \geqs q m$ which yields
\begin{align*}
(k+1) \frac{m-1}{m} 
& = p(m-1) + (q+1) \frac{m-1}{m}
\geqs p(m-1) + q \\
& > p(m-1) + q \left( 1 - \frac1m \right)
= k \frac{m-1}{m}. 
\end{align*}
The implication \eqref{eq:knotmultm} follows. 
\end{proof}

The following result gives a recursion formula for the coefficients $C_{k,n}$. 
First note that 
\begin{equation}\label{eq:C21}
C_{2,1} = m-1. 
\end{equation}

\begin{lemma}\label{lem:recursion}
If $m \geqs 2$ and $k \geqs 2$ then
\begin{equation}\label{eq:recursion2}
C_{k+1,n} = C_{k,n} + C_{k,n-1} 
\Big( (m-1) k - m(n-1) \Big) , \quad 1 \leqs n \leqs \Big\lfloor k \frac{m-1}{m} \Big\rfloor. 
\end{equation}
If $k \geqs 2$ and $k \notin m \no$ then we have also
\begin{equation}\label{eq:recursion3}
C_{k+1,n} = C_{k,n-1} 
\Big( (m-1) k - m(n-1) \Big) , \quad n = \Big\lfloor (k+1) \frac{m-1}{m} \Big\rfloor. 
\end{equation}
\end{lemma}

\begin{proof}
We use \eqref{eq:recursion1} and \eqref{eq:polynomial1}. 
If $k \in m \no$ then the last term in \eqref{eq:polynomial1} is constant. 
Thus \eqref{eq:recursion1} gives
\begin{align*}
p_{\lambda,k+1} (x) 
& = \sum_{n = 0}^{\lfloor k \frac{m-1}{m} \rfloor} 
\lambda^{k+1-n} x^{ (m-1) (k+1) - n m } C_{k,n} \\
& + \sum_{n = 0}^{\lfloor k \frac{m-1}{m} \rfloor - 1} 
\lambda^{k-n} x^{ (m-1)k - n m -1} C_{k,n} \Big( (m-1) k  - nm \Big) \\
& = \sum_{n = 0}^{\lfloor k \frac{m-1}{m} \rfloor} 
\lambda^{k+1-n} x^{ (m-1) (k+1) - n m } C_{k,n} \\
& + \sum_{n = 1}^{\lfloor k \frac{m-1}{m} \rfloor} 
\lambda^{k+1-n} x^{ (m-1) (k+1) - n m } C_{k,n-1} \Big( (m-1) k - m (n-1) \Big). 
\end{align*}
In view of Lemma \ref{lem:multiple} and \eqref{eq:kmultm} this proves \eqref{eq:recursion2} when $k \in m \no$.

If $k \notin m \no$ then $\lfloor k \frac{m-1}{m} \rfloor < k \frac{m-1}{m}$.
In fact if we assume $k \frac{m-1}{m} = p \in \no$ then we get the contradiction
$k = m(k-p) \in m \no$.
Thus the last term in \eqref{eq:polynomial1} contains a positive power of $x$. 
Thus again using \eqref{eq:recursion1} we get
\begin{align*}
p_{\lambda,k+1} (x) 
& = \sum_{n = 0}^{\lfloor k \frac{m-1}{m} \rfloor} 
\lambda^{k+1-n} x^{ (m-1) (k+1) - n m } C_{k,n} \\
& + \sum_{n = 0}^{\lfloor k \frac{m-1}{m} \rfloor} 
\lambda^{k-n} x^{ (m-1) k - n m -1} C_{k,n} \Big( (m-1) k - nm \Big) \\
& = \sum_{n = 0}^{\lfloor k \frac{m-1}{m} \rfloor} 
\lambda^{k+1-n} x^{ (m-1) (k+1) - n m } C_{k,n} \\
& + \sum_{n = 1}^{\lfloor k \frac{m-1}{m} \rfloor + 1} 
\lambda^{k+1-n} x^{ (m-1) (k+1) - n m } C_{k,n-1} \Big( (m-1) k - m (n-1) \Big). 
\end{align*}
Again Lemma \ref{lem:multiple} and \eqref{eq:knotmultm} prove \eqref{eq:recursion2} when $k \notin m \no$, 
and it also shows \eqref{eq:recursion3} when $k \notin m \no$. 
\end{proof}

\begin{rem}\label{rem:coefficient21}
Note that \eqref{eq:C21}, written as $C_{2,1} =  m - 1 = C_{1,1} + C_{1,0} (m-1)$ using
$C_{1,0} = 1$ and forcing $C_{1,1} = 0$, fits into the formula \eqref{eq:recursion2} for $k = n = 1$
(without the upper bound for $n$). Indeed $C_{1,1}$ is not well defined due to 
$\lfloor \frac{m-1}{m} \rfloor = 0$.
\end{rem}

In general it is challenging to find explicit expressions for the coefficients $C_{k,n}$ when $n \neq 0$, 
but $n = 1$ is an exception. 

\begin{lemma}\label{lem:n1}
If $k \geqs 2$ then
\begin{equation}\label{eq:formulakone}
C_{k,1} = \frac12 (m-1) k (k-1).
\end{equation}
\end{lemma}

\begin{proof}
We observe that $C_{2,1} = m-1$ may be written as \eqref{eq:formulakone} for $k = 2$. 
If $k \geqs 3$ then \eqref{eq:indexbounds} implies 
$\lfloor (k-1) \frac{m-1}{m} \rfloor \geqs 1$. 
From $C_{k-1,0} = 1$, Lemma \ref{lem:recursion} and \eqref{eq:recursion2} we get for $k \geqs 3$ and $n = 1$ 
\begin{equation*}
C_{k,1} = C_{k-1,1} + (m-1) (k-1). 
\end{equation*}
If we assume that \eqref{eq:formulakone} holds with $k$ replaced by $k-1$
then we get 
\begin{align*}
C_{k,1} & = C_{k-1,1} + (m-1) (k-1) \\
& = \frac12 (m-1) (k-1) (k-2) + (m-1) (k-1) \\
& = (m-1) (k-1) \left( \frac12 (k-2) + 1 \right) \\
& = \frac12 (m-1) k (k-1). 
\end{align*}
By induction this proves \eqref{eq:formulakone} for all $k \geqs 2$. 
\end{proof}

In the next proposition we need the following lemma. 

\begin{lemma}\label{lem:nonnegativity}
If $m \geqs 2$ and $\theta \geqs \frac2{m}$ then
\begin{equation}\label{eq:functionbound}
f(x) = (1+x)^{m \theta} - \left( 1 - \frac1m \right) x^{m \theta - 1} - 1 \geqs 0
\end{equation}
for all $0 \leqs x \leqs 1$. 
\end{lemma}

\begin{proof}
We have $f(0) = 0$ and $f(1) = 2^{m \theta} - 2 + \frac1m > 0$. 
If $0 < x < 1$ then 
\begin{align*}
f'(x) 
& = m \theta  (1+x)^{m \theta -1} - \left( m \theta  - 1 \right) \left( 1 - \frac1m \right) x^{m \theta  - 2} \\
& = m \theta (1+x)^{m \theta - 2} \left(  1 + x - \left( 1 - \frac{1}{m } \right) \left( 1 - \frac1{m \theta} \right) 
\left( \frac{x}{1+x} \right)^{m \theta - 2}
\right). 
\end{align*}
It follows that $f'(x) > 0$ 
for all $0 < x < 1$. 
The function $f$ is thus strictly increasing in $[0,1]$ which proves 
the claim \eqref{eq:functionbound} for $0 \leqs x \leqs 1$. 
\end{proof}

We may now state and prove our main result on the coefficients $C_{k,n}$, 
which concerns a recursive bound. 

\begin{prop}\label{prop:coefficient}
Suppose $\lambda \in \co \setminus \{ 0 \}$, $m \in \no$, $m \geqs 2$, $\theta \geqs \frac2{m}$,
and consider the polynomials $p_{\lambda, k}$
defined by \eqref{eq:gaussianderivative} 
having the form \eqref{eq:polynomial1} involving 
the coefficients $\{ C_{k,n} \} \subseteq \no \setminus \{ 0 \}$. 
Then we have the bound
\begin{equation}\label{eq:coeffbound}
C_{k,n+1} \leqs C_{k,n} m k^{m \theta}, \quad k \geqs 2, \quad 0 \leqs n \leqs \Big\lfloor k \frac{m-1}{m} \Big\rfloor - 1.
\end{equation}
\end{prop}

\begin{proof}
If $k = 2$ then by \eqref{eq:indexbounds} we have 
$\Big\lfloor k \frac{m-1}{m} \Big\rfloor = 1$.
We have $C_{2,1} = m-1 \leqs C_{2,0} m 2^{m \theta} = m 2^{m \theta}$ so \eqref{eq:coeffbound} is true for $k = 2$. 
In an induction argument we suppose that \eqref{eq:coeffbound} holds for a fixed $k \geqs 2$. 

First Lemma \ref{lem:n1} yields
\begin{equation*}
\frac{C_{k+1,1}}{m (k+1)^{m \theta }} 
= \frac{(m-1) k (k+1)}{ 2 m (k+1)^{m \theta } }
= \frac12 \left( 1 - \frac{1}{m} \right) \frac{k}{(k+1)^{m \theta-1}}
\leqs 1
\end{equation*}
due to the assumption $m \theta \geqs 2$. 
Thus \eqref{eq:coeffbound} holds when $k$ is replaced by $k+1$ and $n=0$. 

Next let $1 \leqs n \leqs \Big\lfloor (k+1) \frac{m-1}{m} \Big\rfloor - 1$. 
If $k \in m \no$ then by Lemma \ref{lem:multiple} and \eqref{eq:kmultm} we have $\Big\lfloor (k+1) \frac{m-1}{m} \Big\rfloor = \Big\lfloor k \frac{m-1}{m} \Big\rfloor$. 
Lemma \ref{lem:recursion} and \eqref{eq:recursion2} gives $C_{k+1,n} \geqs C_{k,n}$ and therefore
by the induction hypothesis
\begin{align*}
\frac{C_{k+1,n+1}}{C_{k+1,n}}
& =   \frac{C_{k,n+1} + C_{k,n} \left( (m-1) k - m n \right) }{C_{k+1,n}} \\
& \leqs \frac{C_{k,n+1}}{C_{k,n}} + (m-1) k \\
& \leqs m k^{m \theta} + (m-1) k \\
& = m (k+1)^{m \theta} \left( \frac{k}{k+1} \right)^{m \theta } \left(  1 + \left( 1 - \frac{1}{m} \right) k^{1-m \theta} \right) \\
& \leqs m (k+1)^{ m \theta }
\end{align*}
in the final inequality using Lemma \ref{lem:nonnegativity}, in the form
\begin{align*}
1 + \left( 1 - \frac1m \right) k^{1-m \theta} 
\leqs \left( \frac{k+1}{k} \right)^{m \theta}
\end{align*}
for all $k \geqs 1$. 
We have now shown the induction step which shows that \eqref{eq:coeffbound}
holds when $k$ is replaced by $k+1$ and $0 \leqs n \leqs \Big\lfloor (k+1) \frac{m-1}{m} \Big\rfloor - 1$, 
provided $k \in m \no$. 

We also need to consider $k \notin m \no$ for which 
Lemma \ref{lem:multiple} and \eqref{eq:knotmultm} yields $\Big\lfloor (k+1) \frac{m-1}{m} \Big\rfloor = \Big\lfloor k \frac{m-1}{m} \Big\rfloor + 1$. 
Thus we assume $1 \leqs n \leqs \Big\lfloor k \frac{m-1}{m} \Big\rfloor$. 

If $1 \leqs n \leqs \Big\lfloor k \frac{m-1}{m} \Big\rfloor - 1$ then 
Lemma \ref{lem:recursion} and \eqref{eq:recursion2} gives $C_{k+1,n} \geqs C_{k,n}$ and thus
by the induction hypothesis
\begin{align*}
\frac{C_{k+1,n+1}}{C_{k+1,n}}
& =   \frac{C_{k,n+1} + C_{k,n} \left( (m-1) k - m n \right) }{C_{k+1,n}} \\
& \leqs \frac{C_{k,n+1}}{C_{k,n}} + (m-1) k \\
& \leqs m k^{m \theta} + (m-1) k \\
& = m (k+1)^{m \theta} \left( \frac{k}{k+1} \right)^{m \theta} \left(  1 + \left( 1 - \frac{1}{m} \right) k^{1-m \theta} \right) \\
& \leqs m (k+1)^{m \theta}
\end{align*}
as before. 

Finally if $n = \Big\lfloor k \frac{m-1}{m} \Big\rfloor$ then again by Lemma \ref{lem:recursion} and \eqref{eq:recursion2} we have $C_{k+1,n} \geqs C_{k,n}$. 
From \eqref{eq:recursion3} we obtain
\begin{equation*}
\frac{C_{k+1,n+1}}{C_{k+1,n}}
\leqs \frac{C_{k+1,n+1} }{C_{k,n}} 
= (m-1) k - m n 
\leqs m (k+1)^{m \theta}
\end{equation*}
again due to the assumption $m \theta \geqs 2$. 
We have shown the induction step in all cases. 
\end{proof}

As a consequence of Proposition \ref{prop:coefficient} we have 
\begin{equation}\label{eq:coeff2bound}
C_{k,2} \leqs m^2 k^{2 m \theta}
\end{equation}
for $k \geqs 4$ and $\theta \geqs \frac2{m}$. In fact \eqref{eq:indexbounds} implies $\Big\lfloor k \frac{m-1}{m} \Big\rfloor - 1 \geqs 1$
if $k \geqs 4$
so Proposition \ref{prop:coefficient} applies to $n = 0$ and $n = 1$.

In the next result we improve the estimate \eqref{eq:coeff2bound}. 

\begin{prop}\label{prop:improvement}
Let $m \in \no$, $m \geqs 2$, $\lambda \in \co \setminus \{ 0 \}$
and consider the polynomials $p_{\lambda, k}$
defined by \eqref{eq:exppoly1} and \eqref{eq:gaussianderivative} 
having the form \eqref{eq:polynomial1} involving 
the coefficients $\{ C_{k,n} \} \subseteq \no \setminus \{ 0 \}$. 
If $k \geqs 4$ then
\begin{equation}\label{eq:improvement}
C_{k,2} \leqs \frac12 m^2 k^{4}.
\end{equation}
\end{prop}

\begin{proof}
From Fa\` a di Bruno's formula \eqref{eq:faadibruno2} we obtain if $k \geqs 1$
\begin{equation}\label{eq:polynomial2}
\begin{aligned}
p_{\lambda,k}(x)
& = \frac{\partial^k g_\lambda(x)}{ g_\lambda (x)}
= \sum_{j=1}^{k} \frac{\lambda^j}{m^j j!} \sum_{\overset{k_1+\dots +k_{j} = k}{1 \leq k_\ell \leq m}} \frac{k!}{k_1! \cdots k_{j}!} \prod_{\ell = 1}^{j} \partial^{k_\ell}_x x^m \\
&= \sum_{j=1}^{k} \lambda^{j} x^{mj - k} \frac{k!}{m^jj!} \sum_{\overset{k_1+\cdots +k_{j} = k}{1 \leq k_\ell \leq m}} \prod_{\ell = 1}^{j} \frac{m!}{(m-k_{\ell})!k_{\ell}!} \\
&= \sum_{n = 0}^{ k-1 } \lambda^{k-n} x^{(m-1) k - nm} \frac{k!}{m^{k-n} (k-n)!} \sum_{\overset{k_1+\cdots +k_{k-n} = k}{1 \leq k_\ell \leq m}} \prod_{\ell = 1}^{k-n} \frac{m!}{(m-k_{\ell})! k_{\ell}!}.
\end{aligned}
\end{equation}
The smallest power of $x$ is $(m-1) k - (k-1)m = m - k$ when $n = k-1$. 
This power seems to be negative if $k > m$ which would be absurd.  
But negative powers are in fact excluded by \eqref{eq:polynomial2}: 
If $k - n = 1$ then the sum over $k_1+\cdots +k_{k-n} = k$ reduces to $k_1 = k$
and $1 \leqs k \leqs m$ is postulated. Thus the smallest power of $x$ is non-negative in \eqref{eq:polynomial2}.

Comparing \eqref{eq:polynomial2} with \eqref{eq:polynomial1} we observe that the upper limits for the summation index $n$ 
seems to be possibly different. In fact by \eqref{eq:indexbounds} we know that  
$\Big\lfloor k \frac{m-1}{m} \Big\rfloor \leqs k - 1$. 
Suppose $n = k-1 > \Big\lfloor k \frac{m-1}{m} \Big\rfloor$. 
Then 
\begin{align*}
k - 2 \geqs \Big\lfloor k \frac{m-1}{m} \Big\rfloor
> k \frac{m-1}{m} - 1 = k - 1 - \frac{k}{m}
\end{align*}
which implies $k > m$. This contradicts the conditions in the sum over $1 \leqs k_1 = k \leqs m$ above, 
which is interpreted as zero then. So in fact \eqref{eq:polynomial1} and  \eqref{eq:polynomial2} are identical. 

The assumption $k \geqs 4$ and \eqref{eq:indexbounds} implies that $n = 2$ has a nonzero coefficient in \eqref{eq:polynomial1}. 
From the identity of \eqref{eq:polynomial1} and  \eqref{eq:polynomial2}
it follows in particular that
\begin{equation*}
C_{k,2} = \frac{k!}{m^{k-2}(k-2)!} \sum_{\overset{k_1+\cdots +k_{k-2} = k}{1 \leq k_\ell \leq m}} \prod_{\ell = 1}^{k-2} \frac{m!}{(m-k_{\ell})!k_{\ell}!}.
\end{equation*}
Since $k_\ell \geqs 1$ for all $1 \leqs \ell \leqs k-2$ and $k_1 + \cdots + k_{k-2} = k$, we have
\begin{equation*}
k_\ell \leq 3, \quad \ell = 1, 2, \ldots, k-2.
\end{equation*}
Hence
\begin{equation}\label{eq:coeffn2}
C_{k,2} = \frac{k!}{m^{k-2}(k-2)!} \sum_{\overset{k_1+\cdots +k_{k-2} = k}{1 \leq k_\ell \leq 3}} \prod_{\ell = 1}^{k-2} \frac{m!}{(m-k_{\ell})!k_{\ell}!}.
\end{equation}

For $k_1, \ldots, k_{k-2}$ satisfying 
\begin{equation*}
k_1 + \cdots + k_{k-2} = k \quad \text{and} \quad k_{\ell} \geq 1,
\end{equation*}
we must have $k_{\ell_1} = 3$ for some $1 \leqs \ell_1 \leqs k-2$, and $k_\ell = 1$ for $\ell \neq \ell_1$ , 
or $k_{\ell_1} = k_{\ell_2} = 2$ for some $1 \leqs \ell_1 \neq \ell_2  \leqs k-2$, 
and $k_\ell = 1$ for $\ell \notin \{ \ell_1, \ell_2 \}$. 

In the first case we have
\begin{equation}\label{eq:firstcase}
\prod_{\ell = 1}^{k-2} \frac{m!}{(m-k_{\ell})!k_{\ell}!}
= \frac{m (m-1) (m-2)}{6} m^{k-3} 
\leqs m^k 
\end{equation}
and in the second case we have 
\begin{equation}\label{eq:secondcase}
\prod_{\ell = 1}^{k-2} \frac{m!}{(m-k_{\ell})!k_{\ell}!}
= \frac{m^2 (m-1)^2}{4} m^{k-4} 
\leqs m^k.
\end{equation}
Combinatorics give
\begin{equation}\label{eq:tamagno}
\sum_{\overset{k_1+\cdots +k_{k-2} = k}{1 \leq k_\ell \leq 3}} \leq k-2 + \binom{k-2}{2} = \frac12 (k-1)(k-2).
\end{equation}

Finally we insert \eqref{eq:firstcase}, \eqref{eq:secondcase} and \eqref{eq:tamagno} into \eqref{eq:coeffn2}
which gives
\begin{equation*}
C_{k,2} \leqs 
\frac{k!}{m^{k-2}(k-2)!}
m^k \frac12 (k-1)(k-2)
= \frac 12 m^2 k (k-1)^2 (k-2) 
\leqs \frac12 m^2 k^4. 
\end{equation*}
\end{proof}

In the next result we look at the particular case of \eqref{eq:exppoly1} where $\lambda = \pm i m$, that is $g_\lambda (x) = e^{ \pm i x^m}$, and the corresponding polynomials $p_{\lambda,k}$ defined by \eqref{eq:gaussianderivative}.

\begin{prop}\label{prop:lowerbound}
Let $m \in \no$, $m \geqs 2$, $\lambda = \pm i m$
and consider the polynomials $p_{\lambda, k}$
defined by \eqref{eq:exppoly1} and \eqref{eq:gaussianderivative} 
having the form \eqref{eq:polynomial1} involving 
the coefficients $\{ C_{k,n} \} \subseteq \no \setminus \{ 0 \}$. 
If $\theta \geqs \frac{2}{m}$ 
then there exists a sequence $\{ k_j \}_{j=1}^{+\infty} \subseteq \no$
such that 
\begin{equation}\label{eq:lowerbound}
\left| p_{\lambda,k_j} ( k_j^\theta) \right| \geqs \frac12 m^{k_j} k_j^{\theta k_j(m-1)}, \quad 
j \in \no \setminus \{ 0 \}. 
\end{equation}
\end{prop}

\begin{proof}
Let $j \in \no \setminus \{ 0 \}$. 
Due to $1 < \frac{m}{m-1} \leqs 2$, the interval $[4 j \frac{m}{m-1}, (4 j+1) \frac{m}{m-1}]$
must contain positive integers. Denote by $k_j \in \no \setminus \{ 0 \}$ the smallest of them. 
Hence we have $4 j \leqs k_j \frac{m-1}{m} \leqs 4 j + 1$, and therefore
\begin{equation}\label{eq:boundskj}
4 j \leqs \Big\lfloor k_j \frac{m-1}{m} \Big\rfloor \leqs 4 j + 1
\end{equation}
as well as 
\begin{equation}\label{eq:indexupperbound}
\Big\lfloor \frac12 \Big\lfloor k_j \frac{m-1}{m} \Big\rfloor \Big\rfloor = 2 j. 
\end{equation}
The bound $k_j \geqs 4 j \frac{m}{m-1} > 4 j$ implies 
that $k_j \geqs 4$ for all $j \in \no \setminus \{ 0 \}$.

For notational simplicity we write $k = k_j$. 
Lemma \ref{lem:derivativepolynomial} and \eqref{eq:polynomial1} gives
\begin{equation*}
p_{\lambda,k} (x) = (\pm 1)^k i^k m^k x^{(m-1) k} \sum_{n = 0}^{\lfloor k \frac{m-1}{m} \rfloor} (\pm 1)^n i^{-n} m^{-n} x^{ -n m} C_{k,n}
\end{equation*}
which in turn gives
\begin{equation}\label{eq:moduluspoly}
\begin{aligned}
\left| p_{\lambda,k} ( k^\theta) \right|
& = m^k  k^{\theta (m-1) k} \left| \sum_{n = 0}^{\lfloor k \frac{m-1}{m} \rfloor} (\pm 1)^n i^{-n} m^{-n} k^{ - \theta n m} C_{k,n} \right| \\
& \geqs m^k k^{\theta (m-1) k} \left| \sum_{ \substack{n = 0 \\ n \ \rm{even} }}^{\lfloor k \frac{m-1}{m} \rfloor} i^{-n} m^{-n}  k^{ - \theta n m} C_{k,n} \right| \\
& = m^k k^{\theta (m-1) k} \left| \sum_{n = 0}^{\lfloor \frac12 \lfloor k \frac{m-1}{m} \rfloor \rfloor} (-1)^n m^{-2n} k^{ - 2 \theta n m}C_{k,2 n} \right|. 
\end{aligned}
\end{equation}

Using \eqref{eq:indexupperbound}, and noting that $m^{-4 j} k^{ - 4 \theta j m }C_{k, 4 j} \geqs 0$, we have
\begin{equation}\label{eq:boundedsum}
\begin{aligned}
& \sum_{n = 0}^{\lfloor \frac12 \lfloor k \frac{m-1}{m} \rfloor \rfloor} (-1)^n m^{-2n} k^{ - 2 \theta n m }C_{k,2 n} 
\geqs \sum_{n = 0}^{ 2j -1} (-1)^n m^{-2n} k^{ - 2 \theta n m }C_{k,2 n} \\
& = 1 - \frac{C_{k,2}}{m^2 k^{2 \theta m}} \\
& \qquad + \frac{C_{k,4}}{m^4 k^{4 \theta m}}  - \frac{C_{k,6}}{m^6 k^{6 \theta m}} + \cdots 
+ \frac{C_{k, 2 (2 j - 2) }}{m^{2 ( 2j - 2) } k^{2 ( 2 j - 2) \theta m}}  - \frac{C_{k, 2 (2j-1) }}{m^{2(2j-1)} k^{2(2j-1) \theta m}}.
\end{aligned}
\end{equation}

By Proposition \ref{prop:coefficient} we have if 
$0 \leqs 2 n \leqs \lfloor k \frac{m-1}{m} \rfloor - 2$,
\begin{equation*}
C_{k,2(n+1)} 
\leqs C_{k,2 n} m^2 k^{2 \theta m}
\end{equation*}
which gives 
\begin{equation*}
\frac{C_{k,2 n }}{m^{2n} k^{2 n \theta m}}  - \frac{C_{k,2(n+1)}}{m^{2(n+1)} k^{2(n+1) \theta m}} \geqs 0
\end{equation*}
for all $n \in \no \setminus \{ 0 \}$ such that $0 \leqs 2 n \leqs \lfloor k \frac{m-1}{m} \rfloor - 2$. 
By \eqref{eq:boundskj} this includes all $n \in \no \setminus \{ 0 \}$ such that $2 \leqs n \leqs 2 j - 2$. 
We may conclude that
\begin{equation}\label{eq:lowerboundsum}
\sum_{n = 0}^{\lfloor \frac12 \lfloor k \frac{m-1}{m} \rfloor \rfloor} (-1)^n m^{-2n} k^{ - 2 \theta n m }C_{k,2 n} 
\geqs 1 - \frac{C_{k,2}}{m^2 k^{2 \theta m}}.
\end{equation}
Finally the assumption $\theta \geqs \frac{2}{m}$ and Proposition \ref{prop:improvement} yield
\begin{equation*}
\frac{C_{k,2}}{m^2 k^{2 \theta m}} \leqs \frac{m^2 k^{4}}{2 m^2 k^{2 \theta m}} = \frac12 k^{2 (2-\theta m)} \leqs \frac12
\end{equation*}
for all $k$. 
Combining this with \eqref{eq:moduluspoly} and \eqref{eq:lowerboundsum}, 
we obtain \eqref{eq:lowerbound}. 
\end{proof}

%%%%%%%%%%%%%%%%%%%%%%%%%%%%%
\section{Proof Theorem \ref{thm:cont1}}\label{sec:proofcont}
%%%%%%%%%%%%%%%%%%%%%%%%%%%%%

In \cite[Theorem~5.7]{Debrouwere1} the authors identify 
the subspace of $g \in C^\infty(\rr d)$ such that the multiplier operator $T_g f = f g$ is continuous on $\Sigma_\theta^s(\rr d)$ and on 
$\mathcal S_\theta^s(\rr d)$, respectively. 
They prove a characterization of the multiplier space 
formulated in terms of estimates of the form
\begin{equation}\label{eq:estimatesmultiplier}
|\partial^\alpha g (x)|
\leqs C h^{|\alpha|} |\alpha|!^s e^{ \lambda^{-\frac{1}{\theta} }|x|^{\frac{1}{\theta} } }, \quad \alpha \in \nn d, \quad x \in \rr d.
\end{equation}
In fact $T_g : \mathcal S_\theta^s(\rr d) \to \mathcal S_\theta^s(\rr d)$ 
is continuous if and only if \eqref{eq:estimatesmultiplier} holds for 
all $\lambda > 0$, some $C = C(\lambda) > 0$, and some $h = h(\lambda) > 0$. 
Moreover $T_g : \Sigma_\theta^s(\rr d) \to \Sigma_\theta^s(\rr d)$
is continuous if and only if \eqref{eq:estimatesmultiplier} holds for 
all $h > 0$, some $C = C(h) > 0$, and some $\lambda = \lambda(h) > 0$. 

\vspace{5mm}

\textit{Proof of Theorem \ref{thm:cont1}.} 

\vspace{3mm}

Theorem \ref{thm:cont1} concerns multiplier functions $g(x) = e^{i q(x)}$ where $q$ is a polynomial of degree $m \geqs 2$, that is 
\begin{equation}\label{eq:qpolynomial1}
q (x) = \sum_{|\alpha| \leqs m} c_\alpha x^\alpha, \quad c_\alpha \in \ro, \quad x \in \rr d,  
\end{equation}
and $c_\alpha \neq 0$ for some $\alpha \in \nn d$ with $|\alpha| = m$. 

If $\gamma \in \nn d$ and $|\gamma| \leqs m$ then
\begin{equation}\label{eq:derivativepoly0}
\partial^\gamma q (x) = \sum_{ \substack{ |\alpha| = m \\ \alpha \geqs \gamma } } 
c_\alpha \frac{ \alpha! }{(\alpha-\gamma)!} x^{\alpha - \gamma} + \rm{L.O.T.} 
\end{equation}
Hence
\begin{equation}\label{eq:estimateqderivative}
|\partial^{\gamma} q(x)| \leq C_{q,m} \max(1, |x|^{m - |\gamma|} ), 
\end{equation}
where $C_{q,m} > 0$ is a constant depending only on $m$ and on the coefficients of $q$.

From Fa\`a di Bruno's formula \eqref{eq:faadibruno2} one gets for $\alpha \in \nn d \setminus \{ 0 \}$ 
\begin{equation}\label{eq:gderivative}
\partial^{\alpha} g(x) = \sum_{j = 1}^{|\alpha|} i^j  \frac{g(x)}{j!} \sum_{\overset{\alpha_1 + \cdots + \alpha_{j} = \alpha}{1 \leq |\alpha_\ell| \leq m, \ 1 \leq \ell \leq j}} \frac{\alpha!}{\alpha_1! \cdots \alpha_j!} \prod_{\ell = 1}^{j} \partial^{\alpha_\ell} q(x). 
\end{equation}
Now \eqref{eq:estimateqderivative} entails if $|x| \geqs 1$
\begin{equation}\label{eq:estimategderivative1}
\begin{aligned}
|\partial^{\alpha} g(x)| 
& \leq  \sum_{j = 1}^{|\alpha|} \frac{1}{j!} \sum_{\overset{\alpha_1 + \cdots + \alpha_{j} = \alpha}{1 \leq |\alpha_\ell | \leq m, \ 1 \leq \ell \leq j}} \frac{\alpha!}{\alpha_1! \cdots \alpha_j!} \prod_{\ell = 1}^{j} C_{q,m} |x|^{m-|\alpha_{\ell}|} \\
&\leq \alpha! \sum_{j = 1}^{|\alpha|} \frac{C^{j}_{q,m}}{j!} |x|^{j m - |\alpha|} \sum_{\overset{\alpha_1 + \cdots + \alpha_{j} = \alpha}{1 \leq |\alpha_\ell| \leq m, \ 1 \leq \ell \leq j}} \frac{1}{\alpha_1! \cdots \alpha_j!}.
\end{aligned}
\end{equation}

On the one hand we get from \cite[Eq.~(0.3.15)]{Nicola1} 
\begin{equation}\label{eq:estimatesize1}
\sum_{\overset{\alpha_1 + \cdots + \alpha_{j} = \alpha}{1 \leq |\alpha_\ell| \leq m, \ 1 \leq \ell \leq j}} \frac{1}{\alpha_1! \cdots \alpha_j!} \leq \binom{m+d}{m}^{j} \leq 2^{(m+d)j}.
\end{equation}
On the other hand if $|x| \geqs 1$ we have for any $\lambda > 0$, using $j \leqs |\alpha| \leqs j m$,
\begin{equation}\label{eq:estimatexpower}
\begin{aligned}
|x|^{j m - |\alpha|}
& = \left( \frac{ \left( \theta^{-1} \left( \lambda^{-1} |x| \right)^{\frac{1}{\theta}} \right)^{j m - |\alpha| } }{( j m - |\alpha| )!} \right)^{\theta} 
( j m - |\alpha| )!^\theta
\left( \theta^\theta \lambda \right)^{j m - |\alpha|} \\
& \leqs e^{ \left( \lambda^{-1} |x| \right)^{\frac{1}{\theta}} }
( j (m - 1) )!^\theta
\left( \theta^\theta \lambda \right)^{j m - |\alpha|}
\end{aligned}
\end{equation}
and, using  \cite[Eq.~(0.3.12)]{Nicola1},
\begin{equation}\label{eq:estimatefactorial}
\begin{aligned}
( j (m - 1) )!^\theta
& \leqs \left( j  (m-1) \right)^{j (m-1) \theta}
= (m-1)^{(m-1) \theta j} j^{j (m-1) \theta} \\
& \leqs \left( e (m-1) \right)^{(m-1) \theta j} j!^{(m-1) \theta}. 
\end{aligned}
\end{equation}

If $|x| \geqs 1$ then insertion of \eqref{eq:estimatesize1}, \eqref{eq:estimatexpower} and \eqref{eq:estimatefactorial} into \eqref{eq:estimategderivative1} gives, 
using the assumption $(m-1)\theta \geq 1$,
\begin{equation}\label{eq:estimategderivative2}
\begin{aligned}
|\partial^{\alpha} g(x)| 
& \leqs \alpha! 
e^{ \left( \lambda^{-1} |x| \right)^{\frac{1}{\theta}} }
\sum_{j = 1}^{|\alpha|} \left( C_{q,m} \left( e (m-1) \right)^{(m-1) \theta}  2^{m+d}  \right)^j  \left( \theta^\theta \lambda \right)^{j m - |\alpha|} j!^{(m-1)\theta-1} \\
& \leqs C^{|\alpha|} |\alpha|!^{(m-1) \theta} e^{ \left( \lambda^{-1} |x| \right)^{\frac{1}{\theta}} } 
\sum_{j = 1}^{|\alpha|} \lambda^{j m - |\alpha|}
\end{aligned}
\end{equation}
for some $C > 0$. 

If $|x| \leqs 1$ then \eqref{eq:estimateqderivative}, \eqref{eq:gderivative}, \eqref{eq:estimatesize1}, 
\eqref{eq:estimatexpower} with $|x| = 1$, and \eqref{eq:estimatefactorial}
give for any $\lambda > 0$
\begin{equation}\label{eq:estimategderivative3}
\begin{aligned}
|\partial^{\alpha} g(x)| 
& \leqs \alpha! \sum_{j = 1}^{|\alpha|} \left( C_{q,m} 2^{m+d} \right)^j j!^{-1} \\
& \leqs \alpha! \sum_{j = 1}^{|\alpha|} \left( C_{q,m} 2^{m+d} \right)^j j!^{-1}
e^{ \lambda^{-\frac{1}{\theta}} }
( j (m - 1) )!^\theta
\left( \theta^\theta \lambda \right)^{j m - |\alpha|} \\
& \leqs \alpha! 
e^{ \lambda^{-\frac{1}{\theta}} }
\sum_{j = 1}^{|\alpha|} \left( C_{q,m} \left( e (m-1) \right)^{(m-1) \theta}  2^{m+d}  \right)^j  \left( \theta^\theta \lambda \right)^{j m - |\alpha|} j!^{(m-1)\theta-1} \\
& \leqs C^{|\alpha|} |\alpha|!^{(m-1) \theta} e^{ \lambda^{-\frac{1}{\theta}} } 
\sum_{j = 1}^{|\alpha|} \lambda^{j m - |\alpha|}
\end{aligned}
\end{equation}
for some $C > 0$. 

The estimate \eqref{eq:estimategderivative2} for $|x| \geqs 1$ and 
the estimate \eqref{eq:estimategderivative3} for $|x| \leqs 1$ can be combined into the estimate for $x \in \rr d$ 
\begin{equation}\label{eq:estimategderivative4}
|\partial^{\alpha} g(x)| 
\leqs e^{ \lambda^{-\frac{1}{\theta}} }  C^{|\alpha|} |\alpha|!^{(m-1) \theta} e^{ \left( \lambda^{-1} |x| \right)^{\frac{1}{\theta}} } 
\sum_{j = 1}^{|\alpha|} \lambda^{j m - |\alpha|}
\end{equation}
for some $C > 0$. 

\vspace{2mm}

\emph{Proof of claim \rm{(i)}}.  

\vspace{2mm}

The assumption $s \geqs (m-1)\theta$ and \eqref{eq:estimategderivative4}
imply 
if $\lambda \geqs 1$ 
\begin{equation}\label{eq:estimategderivative5}
\begin{aligned}
|\partial^{\alpha} g(x)| 
& \leqs e^{ \lambda^{-\frac{1}{\theta}} } C^{|\alpha|} |\alpha|!^s e^{ \left( \lambda^{-1} |x| \right)^{\frac{1}{\theta}} } 
\sum_{j = 1}^{|\alpha|} \lambda^{|\alpha| (m-1)} \\
& \leqs e^{ \lambda^{-\frac{1}{\theta}} } \left( 2 C \lambda^{m-1} \right)^{|\alpha|} |\alpha|!^s e^{ \left( \lambda^{-1} |x| \right)^{\frac{1}{\theta}} } 
 \end{aligned}
\end{equation}
and if $0 < \lambda < 1$
\begin{equation}\label{eq:estimategderivative6}
\begin{aligned}
|\partial^{\alpha} g(x)| 
& \leqs e^{ \lambda^{-\frac{1}{\theta}} } C^{|\alpha|} |\alpha|!^s e^{ \left( \lambda^{-1} |x| \right)^{\frac{1}{\theta}} } 
\sum_{j = 1}^{|\alpha|} \lambda^{- |\alpha|} \\
& \leqs e^{ \lambda^{-\frac{1}{\theta}} } \left( 2 C \lambda^{-1} \right)^{|\alpha|} |\alpha|!^s e^{ \left( \lambda^{-1} |x| \right)^{\frac{1}{\theta}} } 
 \end{aligned}
\end{equation}
Combining \eqref{eq:estimategderivative5} and \eqref{eq:estimategderivative6} we may by the criterion in 
\cite[Theorem~5.7]{Debrouwere1} conclude that $T$ is continuous on $\mathcal S_\theta^s(\rr d)$. 
Claim (i) has been proved.

\vspace{2mm}

\emph{Proof of claim \rm{(ii)}}.  

\vspace{2mm}

The assumptions imply that either $s \geqs \theta (m-1) > 1$ or $s > \theta (m-1) \geqs 1$. 
First we suppose that $s \geqs \theta (m-1) > 1$. 
Then \cite[Theorem~7.1]{Wahlberg1} shows that $T$ is continuous on $\Sigma_\theta^s(\rr d)$.

It remains to consider $s > \theta (m-1) \geqs 1$. Then $\ep := s - \theta (m-1) > 0$. 
From \eqref{eq:estimategderivative4} with $\lambda = 1$ we obtain for any $h > 0$
\begin{align*}
|\partial^{\alpha} g(x)| 
& \leqs e \, C^{|\alpha|} |\alpha|!^{s - \ep} e^{ |x|^{\frac{1}{\theta}} } |\alpha|\\
& \leqs e \left( 2 C h \right)^{|\alpha|} |\alpha|!^s e^{ |x|^{\frac{1}{\theta}} } 
\left( \frac{h^{- \frac{|\alpha|}{\ep} }}{|\alpha|!}   \right)^\ep \\
& \leqs e^{ 1 + \ep h^{- \frac{1}{\ep} } } \left( 2 C h \right)^{|\alpha|} |\alpha|!^s e^{ |x|^{\frac{1}{\theta}} } . 
 \end{align*}
Again by the criterion in 
\cite[Theorem~5.7]{Debrouwere1} we may conclude that $T$ is continuous on $\Sigma_\theta^s(\rr d)$. 
Claim (ii) has been proved. 
\qed

\begin{cor}\label{cor:contdual}
Define $T$ by \eqref{eq:multiplier1}
where $q$ is a polynomial on $\rr d$ with real coefficients and degree $m \geqs 2$, 
and let $s,\theta > 0$. 

\begin{enumerate}[\rm (i)]

\item If $s \geqs (m-1) \theta \geqs 1$ then $T$ is continuous on $\left( \mathcal S_\theta^s \right)'(\rr d)$.

\item If $s \geqs (m-1) \theta \geqs 1$ and $(\theta,s) \neq \left( \frac1{m-1},1 \right)$ then $T$ is continuous on $\left( \Sigma_\theta^s \right)' (\rr d)$.

\end{enumerate}
\end{cor}

%%%%%%%%%%%%%%%%%%%%%%%%%%%%%%%%%%%%%%%%%%%
\section{Proofs of Theorems \ref{thm:discont1} and \ref{thm:discont2}}\label{sec:proofdiscont}
%%%%%%%%%%%%%%%%%%%%%%%%%%%%%%%%%%%%%%%%%%%

\begin{lemma}\label{lem:derivativeestimate}
Let $f,g \in C^\infty(\ro)$. 
If there exists $A,B,a, \theta, \nu > 0$, and $s \geqs 1$ such that 
\begin{equation}\label{eq:derivativeass1}
|D^k f(x)| \leqs B^k k! |f(x)| \eabs{x}^{ k \max \left(\frac{1}{\nu} - 1 , 0 \right) }, \quad k \in \no, 
\end{equation}
and
\begin{equation}\label{eq:derivativeass2}
|D^k (g(x) f(x) )| \leqs A B^k k!^s e^{- a |x|^{\frac1\theta}}, \quad k \in \no, 
\end{equation}
then 
\begin{equation}\label{eq:derivativeconclusion}
|(D^k g)(x) f(x)| \leqs A (2 B)^k k!^s e^{- a |x|^{\frac1\theta}}
\eabs{x}^{ k \max \left(\frac{1}{\nu} - 1 , 0 \right) }, \quad k \in \no. 
\end{equation}
\end{lemma}

\begin{proof}
The assumption \eqref{eq:derivativeass2} for $k=0$ implies \eqref{eq:derivativeconclusion} for $k = 0$. 
In an induction proof we let $k \in \no \setminus \{ 0 \}$ and assume that \eqref{eq:derivativeconclusion} holds for orders smaller than $k$.
We use
\begin{equation*}
(D^k g)(x) f(x)
= D^k ( g(x)  f(x) ) - \sum_{n=0}^{k-1} \binom{k}{n} D^n g (x) D^{k-n} f(x)
\end{equation*}
\eqref{eq:derivativeass1}, \eqref{eq:derivativeass2}, the induction hypothesis, and $s \geqs 1$. 
If $\nu \geqs 1$ then $\max \left( \frac{1}{\nu} - 1,0 \right) = 0$ which gives
\begin{align*}
|(D^k g)(x) f(x)| 
& \leqs |D^k ( g(x)  f(x) )| + \sum_{n=0}^{k-1} \binom{k}{n} |D^n g (x)| \, |D^{k-n} f(x)| \\
& \leqs A B^k k!^s e^{- a |x|^{\frac1\theta}} 
+ \sum_{n=0}^{k-1} \binom{k}{n} |D^n g (x)| \, |f(x)| \, B^{k-n} (k-n)! \\
& \leqs A B^k k!^s e^{- a |x|^{\frac1\theta}} 
+ \sum_{n=0}^{k-1} \binom{k}{n} A (2 B)^n n!^s e^{- a |x|^{\frac1\theta}} B^{k-n} (k-n)! \\
& = A B^k k!^s e^{- a |x|^{\frac1\theta}} 
\left( 1 + 
 \sum_{n=0}^{k-1} \left( \frac{n!}{k!} \right)^{s-1}  2^n
 \right) \\
& \leqs A B^k k!^s e^{- a |x|^{\frac1\theta}} 
\left( 1 + \sum_{n=0}^{k-1} 2^n \right) \\
& = A (2 B)^k k!^s e^{- a |x|^{\frac1\theta}}.  
\end{align*}
This proves the induction step, and hence \eqref{eq:derivativeconclusion}, provided $\nu \geqs 1$.

It remains to consider $0 < \nu < 1$.
Then $\max \left( \frac{1}{\nu} - 1,0 \right) = \frac{1}{\nu} - 1$ and we have  
\begin{align*}
& |(D^k g)(x) f(x)| \\
& \leqs |D^k ( g(x)  f(x) )| + \sum_{n=0}^{k-1} \binom{k}{n} |D^n g (x)| \, |D^{k-n} f(x)| \\
& \leqs A B^k k!^s e^{- a |x|^{\frac1\theta}} 
+ \sum_{n=0}^{k-1} \binom{k}{n} |D^n g (x)| \, |f(x)| \, \eabs{x}^{ (k-n) \left(\frac{1}{\nu} - 1\right) } \, B^{k-n} (k-n)! \\
& \leqs A B^k k!^s e^{- a |x|^{\frac1\theta}} 
+ \sum_{n=0}^{k-1} \binom{k}{n} A (2 B)^n n!^s e^{- a |x|^{\frac1\theta}} 
\eabs{x}^{ k \left(\frac{1}{\nu} - 1\right) }
 B^{k-n} (k-n)! \\
& = A B^k k!^s e^{- a |x|^{\frac1\theta}} 
\left( 1 + \eabs{x}^{ k \left(\frac{1}{\nu} - 1\right) }
 \sum_{n=0}^{k-1} \left( \frac{n!}{k!} \right)^{s-1}  2^n
 \right) \\
& \leqs A B^k k!^s e^{- a |x|^{\frac1\theta}} 
\left( 1 + \eabs{x}^{ k \left(\frac{1}{\nu} - 1\right) } \sum_{n=0}^{k-1} 2^n \right) \\
& = A B^k k!^s e^{- a |x|^{\frac1\theta}} 
\left( 1 + \eabs{x}^{ k \left(\frac{1}{\nu} - 1\right) } 
\left( 2^k -1 \right) \right) \\
& \leqs A (2 B)^k k!^s e^{- a |x|^{\frac1\theta}} 
\eabs{x}^{ k \left(\frac{1}{\nu} - 1\right) } 
\end{align*}
which proves the induction step, and hence \eqref{eq:derivativeconclusion}, when $0 < \nu < 1$.
\end{proof}

Consider the function $\eabs{x}^t = (1+x^2)^{\frac{t}{2}}$ for $x \in \ro$
with $t \in \ro$.
We have the following estimate for its derivatives. 

\begin{lemma}\label{lem:bracket}
If $t \in \ro$
then 
there exists $C_t \geqs 1$ such that 
for all $k \in \no$
\begin{equation}\label{eq:bracketderivative}
\left|  D^k \eabs{x}^t \right| 
\leqs C_t 2^{ 3 k} k! \eabs{x}^{t-k}, \quad x \in \ro. 
\end{equation}
\end{lemma}

\begin{proof}
First we assume $|t| \leqs 1$ and show 
\begin{equation}\label{eq:bracketderivative1}
\left|  D^k \eabs{x}^t \right| 
\leqs 2^{2 k + 5} k! \eabs{x}^{t-k}, \quad x \in \ro, 
\end{equation}
for all $k \in \no$. 

The function $\eabs{x}^t = (1+x^2)^{\frac{t}{2}}$ is a composition of $f(x) = x^{\frac{t}{2}}$
with $g(x) = 1 + x^2$. 
The estimate \eqref{eq:bracketderivative1} is trivial when $k = 0$, and can be verified for $k = 1$ and $k = 2$ so we may assume that $k \geqs 3$.

For $n \in \no$ we have
\begin{equation}\label{eq:derivativef}
f^{(n)}(x) = \frac{t}{2} \left( \frac{t}{2} - 1\right) \cdots \left( \frac{t}{2} - (n-1) \right) x^{\frac{t}{2} - n}
\end{equation}
and if $m_j = 0$ for $3 \leqs j \leqs k$ then
\begin{equation}\label{eq:productg}
\prod_{j=1}^k \left( \frac{g^{(j)} (x)}{j!} \right)^{m_j} = 
\left( g'(x) \right)^{m_1} \left( \frac{g'' (x)}{2}  \right)^{m_2} 
= 2^{m_1} x^{m_1}. 
\end{equation}
If instead $m_j > 0$ for some $3 \leqs j \leqs k$ then the product is zero. 

Inserting \eqref{eq:derivativef} and \eqref{eq:productg} into Fa\`a di Bruno's formula \eqref{eq:faadibruno1} yields
by means of \cite[Eq.~(0.3.5)]{Nicola1}
\begin{align*}
& \frac{ \left|  D^k \eabs{x}^t \right|}{k! \eabs{x}^{t-k}} \\
& \leqs \sum_{m_1 + 2 m_2 = k} \frac{1}{m_1! m_2!}
\left| \frac{t}{2} \left( \frac{t}{2} - 1\right) \cdots \left( \frac{t}{2} - (m_1+m_2-1) \right) \right| 
\eabs{x}^{t - 2(m_1+m_2) + k - t}
2^{m_1} x^{m_1} \\
& \leqs \frac{|t|}{2} \sum_{m_1 + 2 m_2 = k} \frac{1}{m_1! m_2!}
\left( \frac{|t|}{2} +1 \right) \cdots 
\left( \frac{|t|}{2} + m_1 + m_2 + 1 \right)
\eabs{x}^{- (m_1+2m_2) + k} 2^{m_1} \\
& \leqs \sum_{m_1 + 2 m_2 = k} \frac{2^{m_1-1} (m_1+m_2 + 2)!}{m_1! m_2!} 
\leqs \sum_{m_1 + 2 m_2 = k} 2^{2 m_1 + m_2 +1} (m_2+2) (m_2+1) \\
& \leqs \sum_{m_1 + 2 m_2 = k} 2^{2 m_1 + 3 m_2 +4} 
 = \sum_{m_2 = 0}^{\lfloor \frac{k}{2} \rfloor}  2^{2 (k - 2 m_2) + 3 m_2 +4}
= 2^{2 k + 4} \sum_{m_2 = 0}^{\lfloor \frac{k}{2} \rfloor}  2^{- m_2}
\leqs 2^{2 k +5} 
\end{align*}
which proves \eqref{eq:bracketderivative1} provided $|t| \leqs 1$. 

Next we consider $|t| > 1$ and put $m = \lfloor | t | \rfloor$. 
Writing 
\begin{equation*}
\eabs{x}^t = \eabs{x}^{ \pm |t|} = \eabs{x}^{\pm(|t|-m)} \prod_{j=1}^{m} \eabs{x}^{\pm 1}
\end{equation*}
we use Leibniz' rule
\begin{equation*}
D^k \eabs{x}^t 
= \sum_{k_0 + k_1 + \cdots + k_m = k}
\frac{k!}{k_0! k_1 ! \cdots k_m!} D^{k_0} \eabs{x}^{ \pm (|t|-m) } \prod_{j=1}^{m} D^{k_j} \eabs{x}^{\pm 1}. 
\end{equation*}
Since $0 \leqs |t|-m < 1$ we may use \eqref{eq:bracketderivative1} which yields, using \cite[Eq. (0.3.16)]{Nicola1},
\begin{align*}
\left| D^k \eabs{x}^t \right|
& \leqs \sum_{k_0 + k_1 + \cdots + k_m = k}
\frac{k!}{k_0! k_1 ! \cdots k_m!} \left| D^{k_0} \eabs{x}^{ \pm (|t|-m) } \right| \prod_{j=1}^{m} \left| D^{k_j} \eabs{x}^{\pm 1} \right| \\
& \leqs \sum_{k_0 + k_1 + \cdots + k_m = k}
\frac{k!}{k_0! k_1 ! \cdots k_m!} 
2^{2 k_0 + 5} k_0! \eabs{x}^{\pm (|t|-m) - k_0}
\prod_{j=1}^{m} 
2^{2 k_j + 5} k_j! \eabs{x}^{\pm 1- k_j} \\
& \leqs 2^{2 k + 5(m+1)} k! \eabs{x}^{t - k} \sum_{k_0 + k_1 + \cdots + k_m = k} \\
& = 2^{2 k + 5(m+1)} k! \eabs{x}^{t - k} \binom{k+m}{m} \\
& \leqs 2^{3 k + 6 m + 5} k! \eabs{x}^{t - k}
\leqs 2^{3 k + 6 |t| + 5} k! \eabs{x}^{t - k}.
\end{align*}
Combined with \eqref{eq:bracketderivative1} for $|t| \leqs 1$
we have now shown
\eqref{eq:bracketderivative} for all $t \in \ro$. 
\end{proof}

Let $\theta > 0$ and consider the function
\begin{equation}\label{eq:GSfunction1}
f(x) = e^{- \eabs{x}^{\frac{1}{\theta}}}, \quad x \in \ro. 
\end{equation}

\begin{lemma}\label{lem:estimatesGS0}
If $f$ is defined by \eqref{eq:GSfunction1} with $\theta > 0$ then 
there exists $C_\theta \geqs 1$ such that 
for all $k \in \no \setminus \{ 0 \}$
\begin{equation*}
| f^{(k)}(x) | \leqs 
C_\theta^{k} k! f(x) 
\sum_{j=1}^k \frac{1}{j!} \eabs{x}^{j \left(\frac1{\theta} - 1 \right)}.
\end{equation*}
\end{lemma}

\begin{proof}
Fa\`a di Bruno's formula \eqref{eq:faadibruno2},
Lemma \ref{lem:bracket} and \cite[Eq. (0.3.16)]{Nicola1} 
give for some $C_\theta \geqs 1$ if $k \geqs 1$
\begin{align*}
\frac{| f^{(k)}(x) |}{ k ! f(x)}
& = \left| \sum_{j=1}^k \frac{(-1)^j}{j!} \sum_{\overset{k_1 + \cdots + k_j = k}{k_\ell \geqs 1, \ 1 \leq \ell \leq j}} 
\frac{1}{k_1! \cdots k_j!}
\prod_{\ell=1}^j D^{k_\ell} \eabs{x}^{\frac{1}{\theta}} \right| \\
& \leqs \sum_{j=1}^k \frac{1}{j!} \sum_{\overset{k_1 + \cdots + k_j = k}{k_\ell \geqs 1, \ 1 \leq \ell \leq j}} 
\prod_{\ell=1}^j C_\theta 2^{3 k_\ell} \eabs{x}^{\frac{1}{\theta} - k_\ell} \\
& \leqs C_\theta^k 2^{3 k} \sum_{j=1}^k \frac{1}{j!} 
\eabs{x}^{\frac{j}{\theta} - k}
\sum_{\overset{k_1 + \cdots + k_j = k}{k_\ell \geqs 1, \ 1 \leq \ell \leq j}} \\
& \leqs C_\theta^k 2^{3 k} \sum_{j=1}^k \frac{1}{j!} 
\eabs{x}^{\frac{j}{\theta} - k} \binom{k + j - 1}{j - 1} \\
& \leqs C_\theta^k 2^{3 k} \sum_{j=1}^k \frac{1}{j!} 
\eabs{x}^{\frac{j}{\theta} - k} 2^{k + j - 1} \\
& \leqs C_\theta^k 2^{5 k} 
\sum_{j=1}^k \frac{1}{j!} \eabs{x}^{j \left(\frac1{\theta} - 1 \right)}.
\end{align*}
\end{proof}

\begin{cor}\label{cor:estimatesGS1}
If $f$ is defined by \eqref{eq:GSfunction1} with $\theta > 0$ then 
there exists $C_\theta \geqs 1$ such that 
for all $k \in \no$
\begin{equation*}
| f^{(k)}(x) | \leqs 
C_\theta^{k} k! f(x) \eabs{x}^{k \max \left( \frac{1}{\theta} - 1,0 \right)}.
\end{equation*}
\end{cor}

\begin{proof}
The estimate is trivial when $k = 0$. 
Lemma \ref{lem:estimatesGS0} gives for some $C_\theta \geqs 1$ and $k \geqs 1$
with $\tau = \max \left( \frac{1}{\theta} - 1,0 \right)$
\begin{align*}
| f^{(k)}(x) | 
& \leqs C_\theta^{k} k! f(x) \sum_{j=1}^k \frac{1}{j!} \eabs{x}^{j \tau} \\
& \leqs ( 2 C_\theta) ^{k} k! f(x) \eabs{x}^{k \tau}.
\end{align*}
\end{proof}

\begin{prop}\label{prop:expbracket}
If $\theta > 0$ and the function $f$ is defined by \eqref{eq:GSfunction1}
then $f \in \mathcal S_\theta^1(\ro)$. 
\end{prop}

\begin{proof}
If $\theta \geqs 1$ then Corollary \ref{cor:estimatesGS1} gives
\begin{equation*}
| f^{(k)}(x) | \leqs C_\theta^{k} k! e^{- \eabs{x}^{\frac{1}{\theta}}}
\end{equation*}
for some $C_\theta \geqs 1$, and thus it follows from Lemma \ref{lem:seminormequiv} that $f \in \mathcal S_\theta^1(\ro)$. 

Let $0 < \theta < 1$. 
We have for $x \in \ro$ and $j \in \no$
\begin{equation*}
\eabs{x}^j
= \left( \frac{ \eabs{x}^{ \frac{j}{\theta} }}{j!} \right)^\theta j!^\theta
\leqs e^{ \theta \eabs{x}^{ \frac{1}{\theta} }} j!^\theta
\end{equation*}
which combined with Lemma \ref{lem:estimatesGS0} gives for 
some $C_\theta \geqs 1$ and $k \geqs 1$ 
\begin{align*}
| f^{(k)}(x) | 
& \leqs C_\theta^{k} k! f(x) 
\sum_{j=1}^k \frac{1}{j!} \eabs{x}^{j \frac{1-\theta}{\theta}} \\
& \leqs C_\theta^{k} k! f(x) e^{ (1-\theta) \eabs{x}^{ \frac{1}{\theta} }} 
\sum_{j=1}^k j!^{-1 + 1 - \theta} \\
& \leqs ( 2 C_\theta) ^{k} k! e^{-\theta \eabs{x}^{ \frac{1}{\theta} }}. 
\end{align*}
Lemma \ref{lem:seminormequiv} again shows that $f \in \mathcal S_\theta^1(\ro)$.
\end{proof}

We assume that $q$ is a polynomial on $\ro$ with real coefficients of degree $m \geq 2$, that is 
\begin{equation*}
q (x) = \sum_{k = 0}^m c_k x^k, \quad c_k \in \ro, \quad c_m \neq 0, \quad x \in \ro.   
\end{equation*}
In order to show Theorems \ref{thm:discont1} and \ref{thm:discont2}, we may by a rescaling argument assume that $c_m = \pm 1$, 
that is 
\begin{equation}\label{eq:qpolynomial2}
q(x) = \pm x^m + q_{m-1}(x)
\end{equation}
where $q_{m-1}$ is a polynomial with real coefficients of degree $\deg q_{m-1} \leqs m-1$. 

\begin{lemma}\label{lem:reduction}
Let $m \in \no$, $m \geqs 2$, and let $\theta > \frac{2}{m}$. 
Let the polynomial $q$ be defined by \eqref{eq:qpolynomial2}  and let $f \in C^\infty(\ro)$. 
\begin{enumerate}[\rm (i)]

\item Suppose $\theta \geqs 1$ if $m = 3$. If 
\begin{equation}\label{eq:principaltermRoumieu}
e^{ \pm i x^m} f \notin \bigcup_{1 \leqs s < \theta m - \max(\theta,1)} \mathcal S_\theta^s (\ro)
\end{equation}
then 
\begin{equation}\label{eq:polynomialRoumieu}
e^{ i q (x)} f \notin \bigcup_{1 \leqs s < \theta m - \max(\theta,1)} \mathcal S_\theta^s (\ro).
\end{equation}

\item Suppose $\theta > 1$ if $m = 3$. 
If $\frac{2}{m} < \nu < \theta$ 
and 
\begin{equation}\label{eq:principaltermBeurling}
e^{ \pm i x^m} f \notin \bigcup_{1 < s < \nu m - \max(\nu,1) } \Sigma_\theta^s (\ro)
\end{equation}
then 
\begin{equation}\label{eq:polynomialBeurling}
e^{ i q (x)} f \notin \bigcup_{1 < s < \nu m - \max(\nu,1) } \Sigma_\theta^s (\ro).
\end{equation}
\end{enumerate}
\end{lemma}

\begin{proof}
We start by proving statement (i). 
First we discuss the case $m = 2$ in which $q_{m-1} = q_1$ is a polynomial of degree one. 
By assumption $\theta > 1$.
Suppose $e^{ \pm i x^2} f \notin \mathcal S_\theta^s(\ro)$ for some $1 \leqs s < \theta$. 
Then $e^{i q(x)} f = e^{i q_1(x)} e^{ \pm i x^2} f \notin \mathcal S_\theta^s(\ro)$
follows from the invariance of $\mathcal S_\theta^s(\ro)$ with respect to modulation. 
This shows a strengthened form of statement (i). 

It remains to consider the case $m \geqs 3$, in which the assumptions imply that $(m-2) \theta \geqs 1$. 
Assumption \eqref{eq:principaltermRoumieu} means that 
$e^{ \pm i x^m} f \notin \mathcal S_\theta^s (\ro)$ for all 
\begin{equation*}
1 \leqs s < \theta m - \max(\theta,1).
\end{equation*}
In a contradictory argument we suppose that \eqref{eq:polynomialRoumieu} is not true, that is 
$e^{i q} f \in \mathcal S_\theta^s (\ro)$ for some $1 \leqs s < \theta m - \max(\theta,1)$. 

If $(m-2) \theta \leqs s < \theta m - \max(\theta,1)$ then Theorem  \ref{thm:cont1} (i) implies that $e^{ -i q_{m-1}}$ is continuous on $\mathcal S_\theta^s(\ro)$, 
which gives $e^{ \pm i x^m} f = e^{ -i q_{m-1}} e^{i q} f \in \mathcal S_\theta^s(\ro)$. 
This contradicts the assumption \eqref{eq:principaltermRoumieu}. 

If instead $1 \leqs s < (m-2) \theta$ then $e^{i q} f \in \mathcal S_\theta^s (\ro) \subseteq \mathcal S_\theta^{(m-2)\theta} (\ro)$. 
Theorem  \ref{thm:cont1} (i) again implies that $e^{ -i q_{m-1}}$ is continuous on $\mathcal S_\theta^{(m-2)\theta}(\ro)$,
which gives $e^{ \pm i x^m} f = e^{ -i q_{m-1}} e^{i q} f \in \mathcal S_\theta^{(m-2)\theta}(\ro)$. 
This again contradicts the assumption \eqref{eq:principaltermRoumieu}. 
We have shown statement (i) for all $m \geqs 2$.

It remains to prove of statement (ii). 
If $m = 2$ then $q_{m-1} = q_1$ is a polynomial of degree one and $\theta > 1$.
Suppose $e^{ \pm i x^2} f \notin \Sigma_\theta^s(\ro)$ for some $1 < s < \nu$. 
Then $e^{i q(x)} f = e^{i q_1(x)} e^{ \pm i x^2} f \notin \Sigma_\theta^s(\ro)$
follows from the invariance of $\Sigma_\theta^s(\ro)$ with respect to modulation. 
This shows a strengthened form of statement (ii). 

Let $m \geqs 3$. The assumptions imply $(m-2) \theta > 1$. 
Assumption \eqref{eq:principaltermBeurling} means that 
$e^{ \pm i x^m} f \notin \Sigma_\theta^s (\ro)$ for all 
\begin{equation*}
1 < s < \nu m - \max(\nu,1).
\end{equation*}
Suppose that \eqref{eq:polynomialBeurling} is not true, that is 
$e^{i q} f \in \Sigma_\theta^s (\ro)$ for some $1 < s < \nu m - \max(\nu,1)$. 

If $(m-2) \theta \leqs s < \nu m - \max(\nu,1)$ then Theorem  \ref{thm:cont1} (ii) implies that $e^{ -i q_{m-1}}$ is continuous on $\Sigma_\theta^s(\ro)$, 
which gives $e^{ \pm i x^m} f = e^{ -i q_{m-1}} e^{i q} f \in \Sigma_\theta^s(\ro)$. 
This contradicts the assumption \eqref{eq:principaltermBeurling}. 

If instead $1 < s < (m-2) \theta$ then $e^{i q} f \in \Sigma_\theta^s (\ro) \subseteq \Sigma_\theta^{(m-2)\theta} (\ro)$. 
Theorem  \ref{thm:cont1} (ii) again implies that $e^{ -i q_{m-1}}$ is continuous on $\Sigma_\theta^{(m-2)\theta}(\ro)$,
which gives $e^{ \pm i x^m} f = e^{ -i q_{m-1}} e^{i q} f \in \Sigma_\theta^{(m-2)\theta}(\ro)$. 
This again contradicts the assumption \eqref{eq:principaltermBeurling}. 
We have shown statement (ii) for all $m \geqs 2$.
\end{proof}

\vspace{5mm}

\textit{Proof of Theorem \ref{thm:discont1}.} 

\vspace{3mm}

The assumptions for Lemma \ref{lem:reduction} (i) are satisfied. 
If $m = 2$ then the assumptions for Theorem \ref{thm:discont1} are $1 \leqs s < \theta$. 
The proof of \cite[Proposition~2]{Arias1} shows that there exists $f \in \mathcal S_\theta^1 (\ro)$
such that $e^{ \pm i x^2} f \notin \mathcal S_\theta^s (\ro)$ for all $1 \leqs s < \theta$. 
Lemma \ref{lem:reduction} (i) then implies that $e^{ i q(x)} f \notin \mathcal S_\theta^s (\ro)$ for all $1 \leqs s < \theta$
which proves Theorem \ref{thm:discont1} when $m = 2$. 

Suppose next that $m \geqs 3$. 
The assumptions imply $(m-2) \theta \geqs 1$.
First we show that there exists $f \in \mathcal S_\theta^1(\ro)$ such that 
$e^{\pm i x^m} f \notin \mathcal S_\theta^s(\ro)$  
for all $s > 0$ such that $1 \leqs s < \theta m - \max(\theta,1)$.

In fact for the function \eqref{eq:GSfunction1}
we have $f \in \mathcal S_\theta^1(\ro)$ by Proposition \ref{prop:expbracket}.
In order to show $e^{ \pm i x^m} f \notin \mathcal S_\theta^s (\ro)$ for all $s > 0$ such that $1 \leqs s < \theta m - \max(\theta,1)$
we argue by contradiction. 
Let $1 \leqs s < \theta m - \max(\theta,1)$ and suppose that 
$e^{\pm i x^m} f \in \mathcal S_\theta^s(\ro)$. 
By Lemma \ref{lem:seminormequiv}
the estimate \eqref{eq:derivativeass2} with $g(x) = e^{\pm i x^m}$ holds
for some $A, B, a > 0$. By Corollary \ref{cor:estimatesGS1} and Lemma \ref{lem:derivativeestimate}
with $\nu = \theta$
we may conclude that \eqref{eq:derivativeconclusion} holds.  
Thus \eqref{eq:derivativeconclusion} for $x = k^\theta$ gives 
with $\tau = \max \left(\frac1\theta-1, 0\right)$
\begin{equation}\label{eq:estimateabove1}
\begin{aligned}
|(D^k g)( k^\theta) f( k^\theta)| 
& \leqs A (2 B)^k k!^s e^{- a k}
\eabs{ k^{\theta} }^{k \tau} \\
& \leqs A (2 B)^k k!^s e^{- a k}
2^{\frac12 k \tau} k^{k \left( 1 - \min(\theta,1) \right)} \\
& \leqs A (2^{1 + \tau} B)^k k^{k \left( s + 1  - \min(\theta,1)  \right) } e^{- a k }, \quad k \in \no \setminus \{ 0 \}.
\end{aligned}
\end{equation}

On the other hand we may apply 
Proposition \ref{prop:lowerbound}
since the assumptions imply $\theta > \frac2m$.
Hence
\begin{equation}\label{eq:estimatebelow1}
|(D^k g)( k_j^\theta) f( k_j^\theta)| 
= \left| p_{\pm i m,k_j} ( k_j^\theta) ) \right| e^{- \eabs{ k_j^\theta}^{\frac1\theta}}
\geqs \frac12 m^{k_j} k_j^{\theta k_j (m-1)} e^{- 2^{\frac1{2 \theta}} k_j}, \quad j \in \no \setminus\{ 0 \},
\end{equation}
for some sequence  $\{ k_j \}_{j=1}^{+\infty} \subseteq \no$, 
using $k_j^\theta \geqs 1$. 

Combining \eqref{eq:estimatebelow1} and \eqref{eq:estimateabove1} yields 
\begin{equation*}
\frac12 m^{k_j} k_j^{k_j \theta (m-1)} e^{- 2^{\frac1{2 \theta}} k_j}
\leqs A ( 2^{1 + \tau} B)^{k_j} k_j^{k_j \left( s + 1  - \min(\theta,1)  \right) } e^{- a k_j}
\end{equation*}
for $j \in \no \setminus \{ 0 \}$
which contradicts the assumption $s < \theta m - \max(\theta,1)$. 
The assumption that 
$e^{ \pm i x^m} f \in \mathcal S_\theta^s(\ro)$
hence must be wrong. 
Thus 
\begin{equation*}
e^{ \pm i x^m} f \notin \bigcup_{1 \leqs s <  \theta m - \max(\theta,1)} \mathcal S_\theta^s (\ro), 
\end{equation*}
and Lemma \ref{lem:reduction} (i) yields 
\begin{equation*}
T f \notin \bigcup_{1 \leqs s <  \theta m - \max(\theta,1)} \mathcal S_\theta^s (\ro). 
\end{equation*}
Since $f \in \mathcal S_\theta^1 (\ro) \subseteq \mathcal S_\theta^s(\ro)$ when 
$s \geqs 1$
we have in particular 
$T \mathcal S_\theta^s(\ro) \nsubseteq \mathcal S_\theta^s(\ro)$ for any $s > 0$ such that $1 \leqs s < \theta m - \max(\theta,1)$.
\qed

\vspace{5mm}

\textit{Proof of Theorem \ref{thm:discont2}.} 

\vspace{3mm}

Pick $\nu < \theta$ such that $\nu > \frac{2}{m}$, $1 < \nu < \theta$ if $\theta > 1$, 
and $1 < s < \nu m - \max(\nu,1)$.
Set
\begin{equation*}
f(x) = e^{- \eabs{x}^{\frac{1}{\nu}}}, \quad x \in \ro. 
\end{equation*}
By Proposition \ref{prop:expbracket} we have
$f \in \mathcal S_{\nu}^1(\ro)$.

Let $g (x) = g_{\pm i m}(x) = e^{\pm i x^m}$ and 
suppose that 
$T_g f \in \Sigma_\theta^\sigma(\ro)$ for some $\sigma > 0$ such that $1 < \sigma < \nu m - \max(\nu,1)$. 
Then by Lemma \ref{lem:seminormequiv}
the estimate 
\begin{equation*}
|D^k (g(x) f(x) )| \leqs A B^k k!^\sigma e^{- a |x|^{\frac1\theta}}, \quad k \in \no, 
\end{equation*}
holds for all $B, a > 0$, and some $A > 0$ depending on $B$ and $a$. 
By Corollary \ref{cor:estimatesGS1} and Lemma \ref{lem:derivativeestimate}
we may conclude that 
\begin{equation*}
|(D^k g)(x) f(x)| \leqs A (2 B)^k k!^\sigma e^{- a |x|^{\frac1\theta}}
\eabs{x}^{ k \max \left(\frac{1}{\nu} - 1 , 0\right) }, \quad k \in \no, 
\end{equation*}
holds for some $A, B, a > 0$. 
This gives for $x = k^{\nu}$ 
with $\tau = \max \left(\frac1\nu-1,0\right)$
\begin{equation}\label{eq:estimateabove2}
\begin{aligned}
|(D^k g)( k^{\nu}) f( k^{\nu})| 
& \leqs A (2 B)^k k!^\sigma e^{- a k^{\frac{\nu}{\theta}} } 
\eabs{ k^{\nu} }^{k \tau} \\
& \leqs A ( 2^{1 + \tau} B)^k k^{k \left( \sigma + 1  - \min(\nu,1)  \right) } e^{- a  k^{\frac{\nu}{\theta}} }, \quad k \in \no \setminus \{ 0 \}.
\end{aligned}
\end{equation}

Due to $\nu >  \frac2m$
the assumptions for Proposition \ref{prop:lowerbound} with $\theta$ replaced by $\nu$ are satisfied. 
We obtain from 
Proposition \ref{prop:lowerbound}
\begin{equation}\label{eq:estimatebelow2}
|(D^k g)( k_j^{\nu}) f( k_j^{\nu})| 
= \left| p_{\pm i m,k_j} ( k_j^{\nu}) ) \right| e^{- \eabs{ k_j^{\nu}}^{\frac{1}{\nu}}}
\geqs \frac12 m^{k_j} k_j^{{\nu} k_j (m-1)} e^{- 2^{\frac1{2 \nu}} k_j}, \quad j \in \no \setminus \{ 0 \},
\end{equation}
for some sequence  $\{ k_j \}_{j=1}^{+\infty} \subseteq \no$, 
using $k_j^\nu \geqs 1$. 

Combining \eqref{eq:estimatebelow2} and \eqref{eq:estimateabove2} yields finally
\begin{equation*}
\frac12 m^{k_j} k_j^{k_j \nu (m-1)} e^{- 2^{\frac1{2 \nu}} k_j  }
\leqs A (2^{1 + \tau} B)^{k_j} k_j^{k_j \left( \sigma + 1  - \min(\nu,1) \right)} e^{- a k_j^{\frac{\nu}{\theta}} }
\end{equation*}
for $j \in \no \setminus \{ 0 \}$
which contradicts the assumption $\sigma < \nu m - \max(\nu,1)$. 
The assumption that 
$T_g f \in \Sigma_\theta^\sigma(\ro)$
hence must be wrong. 
Thus 
\begin{equation*}
T_g f \notin \bigcup_{1 < \sigma < \nu m - \max(\nu,1)} \Sigma_\theta^\sigma (\ro). 
\end{equation*}
The assumptions for Lemma \ref{lem:reduction} (ii) are satisfied so we obtain 
\begin{equation*}
T f \notin \bigcup_{1 < \sigma < \nu m - \max(\nu,1)} \Sigma_\theta^\sigma (\ro). 
\end{equation*}
Since $f \in \mathcal S_{\nu}^1(\ro) \subseteq \Sigma_\theta^s(\ro)$
and $1 < s < \nu m - \max(\nu,1)$
we get 
$T \Sigma_\theta^s(\ro) \nsubseteq \Sigma_\theta^s(\ro)$. 
\qed

\begin{rem}\label{rem:DebrouwereNeyt}
When $g (x) = g_{\pm i m}(x) = e^{\pm i x^m}$
it is possible to prove the lack of continuity of the operator $T = T_g$ in Theorems \ref{thm:discont1} and \ref{thm:discont2} on 
$\mathcal S_\theta^s (\ro)$ and $\Sigma_\theta^s (\ro)$ respectively, by other means. 
Here we may weaken the assumptions into 
$0 < s < (m-1) \theta$ and $\theta \geqs \frac2m$.
In fact we may use the criterion \eqref{eq:estimatesmultiplier} and \cite[Theorem~5.7]{Debrouwere1}. 

Consider the assumptions of Theorem \ref{thm:discont1} relaxed into $0 < s < (m-1) \theta$ and $\theta \geqs \frac2m$.
By Proposition \ref{prop:lowerbound} we have 
\begin{equation}\label{eq:estimatebelow3}
|(D^k g)( k_j^\theta)| = \left| p_{\pm i m,k_j} ( k_j^\theta) ) \right| 
\geqs \frac12 m^{k_j} k_j^{\theta k_j (m-1)} \quad j \in \no \setminus \{ 0 \}, 
\end{equation}
for some sequence  $\{ k_j \}_{j=1}^{+\infty} \subseteq \no$. 

Suppose that \eqref{eq:estimatesmultiplier} holds 
for all $\lambda > 0$, some $C = C(\lambda) > 0$, and some $h = h(\lambda) > 0$. 
Then we obtain from \eqref{eq:estimatebelow3}
\begin{equation}\label{eq:contradictionestimate}
\frac12 m^{k_j} k_j^{\theta k_j (m-1)}
\leqs C h^{k_j} k_j!^s e^{ \lambda^{-\frac{1}{\theta} } k_j }
\leqs C h^{k_j} k_j^{s k_j} e^{ \lambda^{-\frac{1}{\theta} } k_j }, \quad j \in \no \setminus \{ 0 \}.
\end{equation}
If $s < (m-1) \theta$ this is a contradiction.
By \cite[Theorem~5.7]{Debrouwere1} it follows that $T_g$ is not continuous on $\mathcal S_\theta^s (\ro)$. 

Consider finally the assumptions of Theorem \ref{thm:discont2} relaxed into 
$0 < s < (m-1) \theta$ and $\theta \geqs \frac2m$.
Suppose that \eqref{eq:estimatesmultiplier} holds for 
all $h > 0$, some $C = C(h) > 0$, and some $\lambda = \lambda(h) > 0$. 
Again the estimate \eqref{eq:contradictionestimate} gives a contradiction if $s < (m-1) \theta$. 
It follows by \cite[Theorem~5.7]{Debrouwere1} that $T_g$ is not continuous on $\Sigma_\theta^s (\ro)$. 
\end{rem}

%%%%%%%%%%%%%%%%%%%%%%%%%%%%%

\end{document}